\documentclass{article}

\usepackage[utf8]{inputenc}

\usepackage[
vmargin=24mm,
left=30mm,
right=50mm
]{geometry}

\usepackage{enumitem}
\usepackage{mathtools}  
\usepackage{amssymb}   
\usepackage{amsthm}
\usepackage{amsopn}
\usepackage{stmaryrd}
\usepackage{bbm}

\DeclareMathOperator{\Div}{div}
\DeclareMathOperator*{\prox}{prox}

\usepackage{xcolor}
\usepackage{graphicx}
\usepackage{tikz}
\usetikzlibrary{calc,patterns,shapes.geometric}

\usepackage{pgfplots}
\pgfplotsset{compat=1.16}

\usepackage{multirow}
\usepackage{diagbox}
\usepackage{subcaption}

\usepackage{algorithm}
\usepackage{algpseudocode}

\usepackage{url}
\usepackage{hyperref}
\usepackage{cleveref}

\crefname{hypothesis}{Hypothesis}{Hypotheses}
\crefname{theorem}{Theorem}{Theorems}
\crefname{remark}{Remark}{Remarks}
\crefalias{remark}{remark}
\crefname{fact}{Fact}{Facts}
\crefname{claim}{Claim}{Claims}

\usepackage[backend=biber]{biblatex}
\definecolor{mplolive}{HTML}{bcbd22}
\definecolor{mplblue}{HTML}{1F77B4}
\definecolor{mplgreen}{HTML}{2CA02C}
\definecolor{mplred}{HTML}{D62728}
\definecolor{mplpurple}{HTML}{9467BD}
\definecolor{mplorange}{HTML}{FF7F0E}

\newcommand{\proj}{\operatorname{proj}}
\newcommand{\x}{\mathbf{x}}
\newcommand{\dx}{\mathrm{d}}
\newcommand{\N}{\mathbb{N}}
\newcommand{\R}{\mathbb{R}}

\newcommand{\TGV}{\mathrm{TGV}}
\newcommand{\NSymTGV}{\neg \mathrm{symTGV}}
\newcommand{\BV}{\mathrm{BV}}
\newcommand{\BD}{\mathrm{BD}}
\newcommand{\BGV}{\mathrm{BGV}}
\newcommand{\TV}{\mathrm{TV}}

\newcommand{\OmegaSpace}{\Omega_\mathbf{x}}
\newcommand{\OmegaTime}{\Omega_t}

\newtheorem{theorem}{Theorem}[section]
\newtheorem{definition}[theorem]{Definition}
\newtheorem{lemma}[theorem]{Lemma}
\newtheorem{remark}[theorem]{Remark}
\newtheorem{proposition}[theorem]{Proposition}

\title{Total Generalized Variation for Inverse Problems Involving Piecewise Linear Finite Elements}

\author{
	Moritz Kappes\footnotemark[1]
	\and Manuel Haas\footnotemark[1]
	\and Thomas Beiert\footnotemark[2]
	\and Simone Pezzuto\footnotemark[3]
	\and Alexander Effland\footnotemark[1]
}

\begin{document}
	
	\maketitle
	
	\renewcommand{\thefootnote}{\fnsymbol{footnote}}
	
	\footnotetext[0]{This work was supported by the Deutsche Forschungsgemeinschaft (DFG, German Research
		Foundation), EXC2151-390873048, EXC-2047/1-390685813, and – CRC 1720 – 539309657. This
		project has received funding from the European Research Council (ERC) under the European
		Union’s Horizon Europe research and innovation programme (Grant agreement No. 101230597,
		awarded to S.P.). S.P. also acknowledges support from the CSCS Swiss National Supercomputing
		Centre (project no. lp157) and the SNSF–FWF project “CardioTwin” (no. 214817).}
	\footnotetext[1]{Institute for Applied Mathematics, University of Bonn, Germany}
	\footnotetext[2]{ Heart Center Bonn, Department of Internal Medicine II, University Hospital Bonn, Germany}
	\footnotetext[3]{ Department of Mathematics, University of Trento, Italy}
	
	\renewcommand{\thefootnote}{\arabic{footnote}}
	
\begin{abstract}
	Higher-order regularization has shown improved results over first-order methods such as total variation.
	Piecewise linear discretizations are required in applications such as electrocardiographic imaging, where the forward operator is only defined for weakly differentiable functions.
	In this paper, convergence of a piecewise linear finite element discretization is shown for the non-symmetric second-order total generalized variation functional in a denoising setting.
	Both the function and the auxiliary variable are discretized with piecewise linear finite elements.
	Additionally, this convergence result is extended to a spatiotemporal setting after the spatiotemporal function space is comprehensively introduced.
	In both cases, discrete minimizers converge to the continuous minimizer with rate $h^{1/4}$ in the mesh size $h$, matching the best known rate for finite element discretizations of total variation in the same general setting.
	Numerical experiments on regular and unstructured meshes are consistent with the theoretical convergence rates and indicate faster convergence in practice, with observed rates of $h^{1/2}$ in the spatial and $h$ in the spatiotemporal setting relative to a high-resolution reference solution.
	For image denoising and inpainting on a standard benchmark dataset, reconstruction quality is comparable to related piecewise constant and piecewise linear methods, while using half the degrees of freedom of the piecewise constant discretization.
	The spatiotemporal case is discussed using the example of the inverse problem in electrocardiographic imaging.
	On three synthetic two-dimensional test cases, the proposed functional outperforms Tikhonov and spatiotemporal total variation regularization, reducing the relative error by 12.5 to 43.1\% over first-order Tikhonov regularization.
\end{abstract}

\maketitle

\section{Introduction}
In mathematical inverse problems, one goal is to recover solutions exhibiting refined structural properties, such as smooth transitions and curvature, while preserving sharp edges.
Encoding appropriate assumptions on the solution is therefore essential for obtaining high-quality reconstructions.
In particular, models that incorporate information beyond first-order variations have proven effective to this end.
There is a plethora of methods controlling higher-order derivatives~\cite{Ch97, Ch00, Ch10}, directional information~\cite{Ko19, Pa20, Li24} and curvature~\cite{Sh03} among others.
In addition,  some regularizers can be combined with neural networks to achieve improved reconstruction quality by learning the regularizer~\cite{Ha26, Ko20, Pi21} or its weights~\cite{Vu25}.
One of the simplest regularizers to capture both edges and higher-order smoothness in a unified way is total generalized variation~\cite{Br10}, defined as
\begin{equation*}
	\TGV_\alpha^k(u) \coloneqq \sup \left \lbrace  \int_\Omega u \Div^kv \dx x \colon\Vert \Div^lv\Vert _\infty \leq \alpha_l, l= 0,\ldots,k-1 \right \rbrace.
\end{equation*}
Positive parameters $\alpha$ weigh the different order derivatives of the function, and the supremum is computed over the space of symmetric k-tensors, $v \in \mathcal{C}_c^k(\Omega, \mathrm{Sym}^k(\R^d))$.
Among other important properties of the functional, the kernel of $\TGV_\alpha^k$ contains globally polynomial solutions of order $k-1$~\cite{Br14}, hence overcoming the staircasing effect present in most first-order regularizers.
The most popular version is $\TGV_\alpha^2$, promoting both piecewise constant and piecewise affine solutions.
As with all regularization methods, the question of discretization arises in applications.
While initially discretized on uniform grids, higher demand for applying regularizers on more complex geometries~\cite{Je95, Jo16, Lo06} popularized the use of non-uniform meshes.
Among others, triangular finite element (FE) meshes are a convenient choice.
Using $\TGV$ on FE meshes has been a more recent area of research~\cite{Ba23}.
After applying Fenchel's duality theorem, the formula above can be reformulated as $\TGV_{\alpha}^2 (u)= \min_{w} \alpha_1 \Vert \nabla u - w\Vert +  \alpha_0 \Vert \mathcal{E} w\Vert$
simplifying discretization via the auxiliary variable $w$ and its symmetrized gradient $\mathcal{E} w$.
Example discretizations include graph-based methods~\cite{Go18} and piecewise constant approximations of $u$ with Raviart-Thomas approximations of the auxiliary variable $w$~\cite{Ba23}.

To the best of our knowledge, explicit rates for the convergence of discrete minimizers on unstructured finite element meshes as $h \to 0$ are only available for regularizers leveraging first-order gradient information.
This includes the well-studied total variation (TV), which can be recovered from TGV via $\TGV_\alpha^1(u) = \alpha \TV(u)$.
There are convergence results for TV in FE settings, including the non-convergence of piecewise constant approximations in general~\cite{Ba12}, convergence of piecewise linear approximations with convergence rate $h^{1/4}$~\cite{Ba14}, as well as $\Gamma$-convergence of a non-conforming Crouzeix-Raviart discretization~\cite{Ch20} in image denoising.
The latter is a special case of a more general form to describe mathematical inverse problems, $\min_{u} G(u, z) + \lambda F(u)$, where $G(u, z )$ is a so-called data fidelity term incorporating noisy measurements $z$ and $F(u)$ is a regularizer such as $\TGV$ or $\TV$.
Many problems in imaging (such as denoising, inpainting, and MRI reconstruction), as well as others related to partial differential equations (including electrical impedance tomography, inverse source problems~\cite{Fa25} and the inverse problem in electrocardiographic imaging~\cite{Fr14}), can be cast in such a formulation.

Besides spatial regularization, spatiotemporal regularization has recently been leveraged to advance the state of the art in the medical field~\cite{Bu25, Ha25, On09}.
Usually defined to solve time-dependent PDE problems~\cite[Chapter~6]{Er04}, it can improve solution quality by incorporating temporal dependencies present in the underlying data.
A model problem in this paradigm is the inverse problem of electrocardiographic imaging (ECGI): given sparse electrode measurements on the torso surface, one aims to reconstruct the transmembrane potential inside the heart.
Since this problem is highly ill-posed, the choice of regularization is critical for obtaining meaningful reconstructions~\cite{La26, Wa13}.
TGV is a natural candidate here, as the transmembrane potential combines steep activation fronts with regions of gradual variation, for which TV is prone to staircasing.
Moreover, the transmembrane potential enters the bidomain-based forward model through its gradient, so the forward operator is naturally defined on $H^1$~\cite{Bu10}.
Piecewise constant discretizations of TGV are therefore not admissible in this setting, whereas the piecewise linear discretization developed here is.

In this paper, we show that piecewise linear discretizations of both $u$ and the auxiliary variable $w$ lead to strong convergence of discrete to continuous minimizers w.r.t. the $L^2$ norm in a denoising setting.
More explicitly, we extend the best-known convergence rate for piecewise linear discretizations of $\TV$, $h^{1/4}$, to $\TGV^2$.
As an extension, we present a spatiotemporal version of the regularizer defined on a product domain, and prove the same convergence rate provided an appropriate linear product discretization. 
Numerical experiments confirm the convergence rates and show improved reconstruction results compared to established methods in image reconstruction tasks as well as in ECGI.

The paper is structured as follows:
We begin in \cref{sec:NotationPrelim} by introducing the necessary notation and preliminary results from the literature.
In \cref{sec:SpatialTGVFE}, we present novel theoretical results concerning the convergence of $\TGV$ on a spatial domain in a finite element setting.
After introducing the corresponding finite element space, we define $\TGV$ for spatiotemporal domains in \cref{sec:SpaceTimeReg} and prove convergence properties analogous to those in the spatial case.
Finally, in \cref{sec:NumRes}, we define the primal-dual optimization algorithm used in the numerical calculations and introduce both the spatial problems of image denoising and inpainting as well as the spatiotemporal problem of ECGI.
For all settings, we introduce suitable comparison methods and present both qualitative and quantitative results. 

\section{Notation and Preliminaries}
\label{sec:NotationPrelim}
In this section, we introduce necessary notation and present preliminary results from the literature.

On vectors in $\R^d$, $\ell^p$ norms are denoted as $\vert\cdot\vert_p$, while $\vert\cdot\vert_F$ denotes the Frobenius norm for matrices.
We use $c$ to denote a generic constant that may differ at each occurrence.
We assume that domains $\Omega\subset \R^d$ are bounded and star-shaped.
Without loss of generality, we assume that star-shaped domains are so w.r.t. 0, and for simplicity, we will also assume them to be polygonal.
We define the extension $\Omega_\varepsilon = (1+\varepsilon) \Omega = \lbrace y\in \R^d \colon y = (1+\varepsilon)x, x\in \Omega \rbrace$.
Given a symmetric positive mollifier $\rho$, we define $\rho_\varepsilon(x) \coloneqq \varepsilon^{-d}\rho({x}/{\varepsilon})$.
Additionally, $\Vert f\Vert_{L^p(\Omega)} = \left(\int_\Omega {\sum_{i=1}^{n}\vert f_i(x)\vert^p}\dx x\right)^{1/p}, n=\dim(\mathrm{Im}(f))$.
We adopt the convention that $\Vert\cdot\Vert_{L^p}$ refers to the norm over the whole domain, while any restriction to a subset will be indicated explicitly.

On $\Omega$ we define the space of \emph{k-tensors} and \emph{symmetric k-tensors} as
\begin{align*}
	\mathcal{T}^k(\mathbb{R}^d)
	&= \{ \xi : \underbrace{\mathbb{R}^d \times \cdots \times \mathbb{R}^d}_{k\text{ times}} \to \mathbb{R}
	\colon \xi \text{ is $k$-linear} \}, \\
	\mathrm{Sym}^k(\mathbb{R}^d)
	&= \{ \xi : \underbrace{\mathbb{R}^d \times \cdots \times \mathbb{R}^d}_{k\text{ times}} \to \mathbb{R}
	\colon \xi \text{ is $k$-linear and symmetric}\}.
\end{align*}

We denote the distributional, symmetric Jacobian for vector-valued functions as $\mathcal{E} = 0.5(\nabla + \nabla^T)$, while $\mathcal{M}(\Omega, X)$ is the Banach space of finite signed Radon measures on the set $\Omega$.
We equip $\mathcal{M}(\Omega, X)$ with the Radon norm defined as $\Vert u \Vert_{\mathcal{M}} = \sup \lbrace  \langle u, \phi\rangle \mid \phi \in \mathcal{C}_0(\Omega, X), \Vert\phi\Vert_\infty \leq 1 \rbrace$, where $X$ is either $\mathcal{T}^k(\mathbb{R}^d)$, $\mathrm{Sym}^k(\mathbb{R}^d)$ or their respective matrix-valued versions.
Now, we can define $\TGV$ and its non-symmetric variant $\NSymTGV$~\cite{Br10} for $u\in L^1(\Omega)$ as
\begin{align}
	\label{def:TGV}
	\TGV_\alpha^k(u) &\coloneqq \sup \left \lbrace  \int_\Omega u \Div^kv \dx x\colon v \in \mathcal{C}_c^k(\Omega, \mathrm{Sym}^k(\R^d)),  \Vert \Div^lv\Vert _\infty \leq \alpha_l, {\forall}_{l=0}^{k-1}  \right \rbrace,\\
	\label{def:NSymTGV}
	\NSymTGV_\alpha^k(u) &\coloneqq \sup \left \lbrace  \int_\Omega u \Div^kv \dx x\colon v \in \mathcal{C}_c^k(\Omega, \mathcal{T}^k(\R^d)), \Vert \Div^lv\Vert _\infty \leq \alpha_l\ {\forall}_{l=0}^{k-1} \right \rbrace.
\end{align}
Throughout this paper, we will mainly consider order $k=2$.
By using Fenchel's duality theorem~\cite[Theorem~3.5]{Br14}, we can reformulate the functionals above to
\begin{align}
	\TGV_{\alpha}^2 (u)= \min_{w\in \mathcal{M}(\Omega, \R^d)} \alpha_1 \Vert \nabla u - w\Vert_{\mathcal{M}(\Omega, \R^d)} +  \alpha_0 \Vert \mathcal{E} w\Vert _{\mathcal{M}(\Omega,\mathbb{R}^{d \times d}_{\mathrm{sym}})},\label{eq:characterizationTGV}\\
	\NSymTGV_{\alpha}^2 (u)= \min_{w\in \mathcal{M}(\Omega, \R^d)} \alpha_1 \Vert \nabla u - w\Vert _{\mathcal{M}(\Omega, \R^d)} +  \alpha_0\Vert \nabla w\Vert _{\mathcal{M}(\Omega, \R^{d\times d})}.\label{eq:characterizationNsymTGV}
\end{align}
It was shown in~\cite[Theorem~3.5]{Br14} that the regularity of the auxiliary variable $w$ must be higher since measure-valued minimizers generally have unbounded distributional gradients and would hence result in the functional assuming the value $+ \infty$.
Hence, the minimization problems are instead solved over the spaces of bounded deformations and bounded variation defined as
\begin{align*}
	\BD(\Omega, \R^d) &\coloneqq \left\lbrace u \in L^1(\Omega, \R^d) \colon \mathcal{E}(u)\in \mathcal{M}(\Omega, \mathbb{R}^{d \times d}_{\mathrm{sym}})\right\rbrace,\\
	\BV(\Omega, \R^d) &\coloneqq \left\lbrace u \in L^1(\Omega, \R^d) \colon \nabla u\in \mathcal{M}(\Omega, \mathbb{R}^{d \times d})\right\rbrace.
\end{align*}
For more details about these spaces, we refer the reader to~\cite{Am00, Te80}.
Both the symmetric and the non-symmetric TGV functionals are seminorms on $L^1(\Omega)$ and therefore give rise to the Banach spaces of \emph{(non-symmetric) bounded generalized variation}, $\BGV_\alpha^k(\Omega)$, endowed with the norm $\Vert u\Vert_{\BGV_\alpha^k(\Omega)} \coloneqq \Vert u\Vert_{L^1(\Omega)} + \TGV_\alpha^k(u)$ (see~\cite[Proposition~3.5]{Br10}).
Further, it was shown in~\cite[Corollary~3.13]{Br14} that the full norms are equivalent for any $\alpha, k > 0$.

In the first half of the paper, we will leverage approximation results of the classical total variation regularizer, $\TV(u) \coloneqq \Vert D u\Vert_{\mathcal{M}}$.
By~\cref{def:TGV} and~\cref{def:NSymTGV}, we can see that $\TV(u)= \TGV_{1}^1(u) = \NSymTGV_{1}^1(u)$.
Functions $u\in \BV(\Omega, \R^d)$ can be approximated as follows:
\begin{proposition}[{\protect\cite[Proposition~2.1]{Ba14}}]
	\label[proposition]{prop:approxBV}
	Assume that $\Omega$ is bounded, star-shaped and $\partial \Omega \in \mathcal{C}^{0,1}$. If $u\in \mathrm{BV}(\Omega)$, then for every  $\varepsilon > 0$ there exists a $u_\varepsilon \in \mathcal{C}^\infty(\Omega)$ such that
	\begin{equation*}
		\Vert u - u_\varepsilon\Vert_{L^1} \leq \varepsilon \Vert D u\Vert_{\mathcal{M}},\quad  \Vert \nabla u_\varepsilon\Vert_{L^1} \leq (1 + c \varepsilon) \Vert Du\Vert_{\mathcal{M}},\quad \Vert D^2u_\varepsilon\Vert_{L^1} \leq {c/\varepsilon} \Vert D u\Vert_{\mathcal{M}}.
	\end{equation*}
\end{proposition}

We will carry out both space and time discretization with finite element techniques.
Given $\Omega$ we use a conforming triangulation $\mathcal{T}_h$ of the domain such that $\overline{\Omega} = \cup_{T\in\mathcal{T}_h}\overline{T}$.
Throughout this paper, we will mainly consider functions that are piecewise linear or constant on each $T\in\mathcal{T}_h$; the respective FE spaces are defined as
\begin{align}
	\mathcal{V}_h^k &\coloneqq \left\lbrace v\in C(\bar{\Omega})^k \colon v_i\lvert_{T}\in \mathcal{P}_1(T), \forall T \in \mathcal{T}_h, i=1, \ldots, k \right\rbrace\label{eq:pw_linFEspace},\\
	\mathcal{Q}_h^k &\coloneqq \left\lbrace q\in L^1(\Omega)^k \colon q_i\lvert_{T}\in \mathcal{P}_0(T), \forall T \in \mathcal{T}_h, i=1,\ldots, k \right\rbrace\label{eq:pw_constFEspace}.
\end{align}
Both spaces above are conforming, i.e. $ \mathcal{V}_h^k \subset H^1(\Omega)^k$ and $\mathcal{Q}_h^k\subset L^2(\Omega)^k$.
Due to the finite dimensionality of the FE space and the resulting finite number of jumps in the gradient, it also holds that $\mathcal{V}_h^k\subset\BV(\Omega, \R^d)$ for all $k\geq 1$; the same holds for $\BD(\Omega, \R^d)$.
Piecewise constant functions are, in general, discontinuous across interior edges, so we define the jump across interior edges as $\llbracket u \rrbracket \coloneqq u_+ - u_- = u\rvert_{T_+} - u\rvert_{T_-}$.
To achieve a consistent notation, we choose an arbitrary but fixed orientation, denoting one neighboring triangle by $T_+$ and the other by $T_-$, and fix outer unit normals $\nu_+, \nu_-$.
Additionally, the discretization parameter $h = \max_{T\in \mathcal{T}_h}\mathrm{diam}(T)$ is defined as the maximal diameter of the triangulation, and $X_{E,i}$ denotes the endpoints of each edge $E$.

In addition to the spaces above, comparison methods employ the zero-order Raviart-Thomas FE space defined as
\begin{equation*}
	\mathcal{RT}_0(\Omega) \coloneqq \left\lbrace v\in H(\Div,\Omega) \colon v\rvert_T \in \mathcal{P}_0(T)^2 + \begin{pmatrix}
		x\\y
	\end{pmatrix}\mathcal{P}_0(T) \, \forall T\in \mathcal{T}_h\right\rbrace,
\end{equation*}
where $H(\Div,\Omega) \coloneqq \left\lbrace v \in L^2(\Omega, \R^2) \colon \Div v \in L^2(\Omega)\right\rbrace$.
A piecewise polynomial function belongs to $H(\Div,\Omega)$ if and only if its normal component is continuous across each edge, which is the defining feature of the Raviart-Thomas FE space.

The convergence analysis relies on approximation properties of FE interpolation operators.
For the nodal Lagrange interpolation, which maps a function $u\in\mathcal{C}^2(\Omega)$ onto its values on the nodes of $\mathcal{T}_h$ the following remark holds~\cite{Ci78}:
\begin{remark}
	\label[remark]{rem:FEinterpolationIneq}
	Given the nodal Lagrange interpolation operator $\mathcal{I}_h, u\in\mathcal{C}^2(\Omega)$ and an arbitrary triangle $T \in\mathcal{T}_h$,
	\begin{equation}
		\label{eq:FEinterpolationIneq}
		\Vert u - \mathcal{I}_hu\Vert _{L^p(T)} + h\Vert \nabla(u - \mathcal{I}_hu)\Vert _{L^p(T)} \leq ch^2 \Vert D^2u\Vert _{L^p(T)}.
	\end{equation}
\end{remark}
The same estimate holds for the Clément interpolation~\cite{Cl75} of $H^2$ functions upon replacing the domain of integration on the right-hand side, $T$, by its neighborhood.
Additionally, it extends to the vector-valued case by applying it componentwise.
We later use prismatic elements to discretize the spatiotemporal finite element space in~\cref{sec:SpaceTimeReg}.
There, the same approximation rate w.r.t. $h$ holds as the estimate extends to general affine elements~\cite{Er21}.
Since this is the only property we require from the interpolation operators in this work, they are uniformly denoted as $\mathcal{I}_h$.

\section{Spatial Total Generalized Variation on Finite Elements}
\label{sec:SpatialTGVFE}
To develop a spatiotemporal $\TGV$ framework for PDE-constrained inverse problems such as the inverse problem in electrocardiographic imaging, we first analyze the corresponding purely spatial setting.
This spatial analysis serves as the foundation for the later spatiotemporal theory and allows us to first provide results on the purely spatial case.
In particular, we extend the best-known convergence rate for denoising regularized with $\TV$ to $\TGV_\alpha^2$.
We first consider the case in which the auxiliary variable $w$ is continuous and then pass to the fully discretized setting.\label{subsec:disc}

\subsection{Convergence of Semidiscrete TGV}
First, we establish some intermediate results in the case where only the function $u$ is discretized. In this section, we limit ourselves to $\TGV_\alpha^2$; however, all results also hold for $\NSymTGV_\alpha^2$.
To discretize $\TGV_\alpha^2$ and $\NSymTGV_\alpha^2$, we use the alternative characterizations~\cref{eq:characterizationTGV} and~\cref{eq:characterizationNsymTGV}.

\begin{lemma}
	\label{lem:BoundDiscreteTGVdiff}
	Let $\mathcal{T}_h$ be an admissible triangulation of $\Omega$, $u \in \mathcal{C}^2(\Omega)$ and $\mathcal{I}_h$ the Lagrange or Clément interpolation operator onto $ \mathcal{V}_h^k$.
	Then $\TGV_{\alpha}^2(u - \mathcal{I}_h u) \leq \alpha_1 \sum_{T \in \mathcal{T}_h} \Vert \nabla(u - \mathcal{I}_h u)\Vert _{L^1(T)}$.
\end{lemma}
\begin{proof}
	Using the representation in \eqref{def:TGV}, together with integration by parts and Hölder's inequality, we obtain:
	\begin{align*}
		&\TGV_{\alpha}^2(u - \mathcal{I}_h u)\\
		&\quad= \sup \left \lbrace  \int_\Omega(u - \mathcal{I}_h u)  \Div^2v \dx x\colon v \in \mathcal{C}_c^2(\Omega, \mathrm{Sym}^2(\R^d)),  \Vert v\Vert _\infty \leq \alpha_0, \Vert \nabla \cdot v\Vert _\infty \leq \alpha_1\right \rbrace\\
		&\quad= \sup \left \lbrace  \int_\Omega \nabla(\mathcal{I}_h u - u)  \nabla \cdot v \dx x\colon v \in \mathcal{C}_c^2(\Omega, \mathrm{Sym}^2(\R^d)), \Vert v\Vert _\infty \leq \alpha_0, \Vert \nabla \cdot v\Vert _\infty \leq \alpha_1\right \rbrace\\
		&\quad \leq \alpha_1 \sum_{T \in \mathcal{T}_h} \Vert \nabla(u - \mathcal{I}_h u)\Vert _{L^1(T)}.
	\end{align*}
\end{proof}
Combining this result with~\cref{prop:approxBV} yields a bound on the total generalized variation of the interpolation error of the smooth approximation $u_\varepsilon$ in terms of $u$.
Additionally, we need to control its total generalized variation directly:
\begin{lemma}
	\label{lem:approxTGV}
	Let $u$ be a function such that $\TGV_\alpha^2(u)<\infty$, then for its approximation $u_\varepsilon$ as in~\cref{prop:approxBV} it holds that
	$
	\TGV_\alpha^2(u_\varepsilon) \leq (1 + c\varepsilon) \TGV_\alpha^2(u).
	$
\end{lemma}
\begin{proof}
	We use the explicit construction of $u_\varepsilon$:
	For $x \in \Omega_\varepsilon$ we define $\tilde{u}_\varepsilon(x) = u(\frac{x}{1+\varepsilon})$ and $u_\varepsilon(x) = \left(\tilde{u}_\varepsilon * \rho_\varepsilon \right)(x)$.
	Then we extend $u$ by zero on $\R^d$ and calculate:
	\begin{align*}
		&\TGV_{\alpha}^2(u_\varepsilon)=\sup \left \lbrace  \int_\Omega u_\varepsilon  \Div^2v \dx x\colon v \in \mathcal{C}_c^2(\Omega, \mathrm{Sym}^2(\R^d)), \Vert v\Vert _\infty \leq \alpha_0, \Vert \nabla \cdot v\Vert _\infty \leq \alpha_1\right \rbrace\\
		&\quad\leq \sup \left \lbrace  \int_{\Omega_\varepsilon} \tilde{u}_\varepsilon \Div^2(v * \rho_\varepsilon) \dx x\colon v \in \mathcal{C}_c^2(\Omega, \mathrm{Sym}^2(\R^d)), \Vert v\Vert _\infty \leq \alpha_0, \Vert \nabla \cdot v\Vert _\infty \leq \alpha_1\right \rbrace\\
		&\quad\leq \sup \left \lbrace  \int_{\Omega_\varepsilon} \tilde{u}_\varepsilon \Div^2v \dx x\colon v \in \mathcal{C}_c^2(\Omega_\varepsilon, \mathrm{Sym}^2(\R^d)), \Vert v\Vert _\infty \leq \alpha_0, \Vert \nabla \cdot v\Vert _\infty \leq \alpha_1\right \rbrace\\
		&\quad\leq (1+\varepsilon)^d \TGV_{\alpha}^2(u) \leq (1 + c\varepsilon) \TGV_{\alpha}^2(u),
	\end{align*}
	using the symmetry of the mollifier and Fubini’s theorem to move the mollification inside the divergence. Since the mollification operator then commutes with the divergence, we invoke its non-expansiveness w.r.t. the norm together with the scaling properties of $\TGV_\alpha^2$~\cite[Proposition~3.3]{Br10} to conclude the proof.
\end{proof}

With these intermediate results established, we can show convergence of a $\TGV_\alpha^2$-regularized inverse problem akin to~\cite[Theorem~7.1]{Ba14}.
Let $\psi(u) =  \TGV_{\alpha}^2(u) + \frac{\lambda}{2}\Vert u - z\Vert _{L^2}^2$ be the \emph{objective functional}, where $u\in L^2(\Omega)$ and $z\in L^\infty(\Omega)$ is the \emph{noisy input} function.
Since $\psi$ is strictly convex, there exists a unique $\xi \in L^2(\Omega),\TGV_{\alpha}^2(\xi) < \infty$ and due to the finite dimensionality of $\mathcal{V}_h^1$ a unique $\xi_h \in \mathcal{V}_h^1$ such that 
\begin{equation}
	\label{def:contTGVdenoiseMin}
	\psi(\xi) = \inf \lbrace \psi(u) \rvert u \in L^2(\Omega)\rbrace, \quad \psi(\xi_h) = \inf \lbrace \psi(u) \rvert u \in \mathcal{V}_h^1\rbrace  < \infty.
\end{equation}

\begin{theorem}
	\label{thm:semi_discrete_conv}
	Assume that $\Omega$ is star-shaped. Let $\xi \in L^\infty(\Omega)$, $\xi_h \in \mathcal{V}_h^1$ be defined as in~\cref{def:contTGVdenoiseMin} and $\alpha = (\alpha_0, \alpha_1), \alpha_0, \alpha_1>0$. Then
	$
	\Vert \xi - \xi_h\Vert_{L^2} \leq ch^{1/4} \left( \Vert\xi\Vert_{L^1} +  \TGV_{\alpha}^2(\xi)\right)^{1/2}.
	$
\end{theorem}
\begin{proof}
	Let $\varepsilon > 0 $ and $\xi_\varepsilon \in \mathcal{C}^\infty(\Omega)$ be an approximation that satisfies all the properties of~\cref{prop:approxBV}. Since $\psi$ is strongly convex with parameter ${\lambda}$ and $\xi_h$ is a discrete minimizer, we have
	\begin{align*}
		\frac{\lambda}{2}\Vert \xi& - \xi_h\Vert_{L^2}^2 \leq \psi(\xi_h) - \psi(\xi) \leq \psi(\mathcal{I}_h\xi_\varepsilon) - \psi(\xi)\\
		&= ( \TGV_{\alpha}^2(\mathcal{I}_h\xi_\varepsilon) - \TGV_{\alpha}^2(\xi)) +  \frac{\lambda}{2}(\Vert \mathcal{I}_h\xi_\varepsilon - z\Vert_{L^2}^2 -  \Vert \xi - z\Vert_{L^2}^2) = \mathcal{A}_1 + \mathcal{A}_2.
	\end{align*}
	Here $\mathcal{I}_h$ is the Clément-interpolation operator onto $\mathcal{P}^1$ functions.
	
	$\mathcal{A}_1$: We add and subtract the $\TGV_{\alpha}^2$ seminorm  and use~\cref{lem:BoundDiscreteTGVdiff},~\cref{lem:approxTGV} and~\cref{prop:approxBV} as well as~\cref{rem:FEinterpolationIneq}  to bound the first summand, 
	\begin{align*}
		\mathcal{A}_1 &\leq \TGV_{\alpha}^2(\mathcal{I}_h\xi_\varepsilon - \xi_\varepsilon) + \TGV_{\alpha}^2(\xi_\varepsilon) - \TGV_{\alpha}^2(\xi)\leq \frac{ch}{\varepsilon} \mathrm{TV}(\xi) + c\varepsilon \TGV_{\alpha}^2(\xi).
	\end{align*}
	
	$\mathcal{A}_2$: We use approximation properties of $\mathcal{I}_h$ and the fact that $\xi$ and $z$ are essentially bounded by some constant $c$,
	\begin{equation*}
		\mathcal{A}_2 \leq \frac{c\lambda}{2}\left( \Vert \mathcal{I}_h\xi_\varepsilon - \xi_\varepsilon\Vert_{L^1} + \Vert \xi_\varepsilon - \xi\Vert_{L^1}\right) \leq c\left( \frac{h^2}{\varepsilon} + \varepsilon\right)  \mathrm{TV}(\xi).
	\end{equation*}
	
	Using the equivalence of the full norms and combining both summands gives the estimate $\Vert \xi - \xi_h\Vert_{L^2}^2 \leq c\left( \frac{h^2}{\varepsilon}  + \frac{h}{\varepsilon}+ \varepsilon\right)  (\TGV_{\alpha}^2(\xi) + \Vert \xi\Vert_{L^1})$.
	After choosing $\varepsilon = h^{1/2}$ and removing higher-order terms, we conclude that $\Vert \xi - \xi_h\Vert_{L^2} \leq ch^{1/4}  (\TGV_{\alpha}^2(\xi) + \Vert \xi\Vert_{L^1})^{1/2}$.
\end{proof}

\subsection{Convergence of Fully Discrete, Non-symmetric TGV}
While the preceding section extends the known convergence rates for $\TV$ to $\TGV_\alpha^2$, we cannot expect to observe these rates in numerical experiments yet.
To numerically compute $\TGV^2_{\alpha}(u)$, the minimization w.r.t. the auxiliary variable $w$ must be carried out, which in turn requires its discretization.
We know from \cref{sec:NotationPrelim} that for $k=2$, $w\in\BD(\Omega, \R^d)$. Hence, we would need a discretization of vector-valued elements that converge to $w$ in the bounded deformation norm.
To the best of our knowledge, convergence results only exist on subspaces of $\BD(\Omega, \R^d)$ and with specialized finite elements~\cite{Bab23, Co19}.
Hence, we will consider $\NSymTGV_\alpha^2$ as a regularizer in the following, which requires convergence in the bounded variation norm.

As before we have $u_h \in \mathcal{V}_h^1$ whereas $w_h \in  \mathcal{V}_h^2$, the vector valued $\mathcal{P}^1$ finite element space.
For such $u_h, w_h$ we can reformulate \eqref{eq:characterizationNsymTGV} to 
\begin{equation} 
	\label{eq:fullydiscretetgv}
	\NSymTGV_{h, \alpha}^2 (u_h)= \min_{w_h\in \mathcal{V}_h^2} \sum_{T\in\mathcal{T}_h} \alpha_1 \Vert \nabla u_h - w_h\Vert_{L^1(T)} + \alpha_0\Vert  \nabla w_h\Vert_{L^1(T)},
\end{equation}
and define $\psi_h(u) \coloneqq \NSymTGV_{h, \alpha}^2(u) + \frac{\lambda}{2}\Vert u - z\Vert _{L^2}^2$.
Similarly to the last section we use $\psi(u) \coloneqq \NSymTGV_{\alpha}^2(u) + \frac{\lambda}{2}\Vert u - z\Vert _{L^2}^2$.
Due to the strict convexity of the objective functional and the finite-dimensionality of the FE space, we again have unique $\xi \in L^2(\Omega), \xi_{h} \in \mathcal{V}_h^1$ such that
\begin{equation}
	\label{def:discTGVdenoiseMin}
	\psi(\xi) = \inf \lbrace \psi(u) \rvert u \in L^2(\Omega)\rbrace,\quad \psi_h(\xi_{h}) = \inf \lbrace \psi_h(u) \rvert u \in \mathcal{V}_h^1\rbrace  < \infty.
\end{equation}
To generalize the convergence results for the fully discretized case, we must quantify the effect of discretizing $w$.
To this end, we observe that the value of the semidiscrete regularizer does not exceed that of the fully discrete regularizer.
\begin{remark}
	\label[remark]{rem:BoundDiscTGV}
	Due to $\mathcal{V}_h^2\subset \BV(\Omega,\R^d)$, it holds that
	\begin{align*}
		&\NSymTGV_{\alpha}^2 (u_h) = \min_{w\in \BV(\Omega, \R^d)} \alpha_1  \sum_T \Vert \nabla u_h - w\Vert_{L^1(T)} + \alpha_0\Vert \nabla w\Vert_{\mathcal{M}(\Omega, \R^{d \times d})}   \\
		&\leq \min_{w_h\in \mathcal{V}_h^2} \sum_T \alpha_1 \Vert \nabla u_h - w_h\Vert_{L^1(T)} + \alpha_0\Vert  \nabla w_h\Vert_{L^1} = \NSymTGV_{h, \alpha}^2 (u_h).
	\end{align*}
	
\end{remark}
Additionally, the converse inequality holds with an error term diminishing in $h$.
\begin{lemma}
	\label{lem:discTGVboundAbove}
	Let $u_h \in \mathcal{V}_h^1$ be arbitrary. 
	Then it holds that
	\begin{equation*}
		\NSymTGV_{h,\alpha}^2(u_{h}) \leq (1 + c(\varepsilon + \frac{h}{\varepsilon} + \frac{h^2}{\varepsilon}))\NSymTGV_{\alpha}^2(u_{h}).
	\end{equation*}
\end{lemma}
\begin{proof}
	Let $w_\varepsilon$ be the continuous approximation from~\cref{prop:approxBV}, defined componentwise, using the equivalence of norms on $\R^d$ to extend the argument to the vector-valued setting.
	Let $\mathcal{I}_h$ denote the Clément interpolation onto $\mathcal{V}_h^2$.
	Using the minimality of $w_h$ and that $\mathcal{I}_hw_\varepsilon\in\mathcal{V}_h^2 $ gives:
	\begin{align*}
		&\NSymTGV_{h,\alpha}^2(u_{h}) = \min_{w_h\in \mathcal{V}_h^2} \sum_T \alpha_1 \Vert \nabla u_h - w_h\Vert_{L^1(T)} + \alpha_0\Vert  \nabla w_h\Vert_{L^1} \\
		&\quad\leq \min_{w\in \BV(\Omega, \R^d)} \sum_T \alpha_1 \Vert \nabla u_h - w_\varepsilon\Vert_{L^1(T)} + \alpha_0\Vert  \nabla w_\varepsilon\Vert_{L^1} \\
		&\qquad + \sum_T \alpha_1 \Vert w_\varepsilon - \mathcal{I}_hw_\varepsilon\Vert_{L^1(T)} + \alpha_0\Vert \nabla(w_\varepsilon - \mathcal{I}_h w_\varepsilon)\Vert_{L^1}\\
		&\quad\leq \min_{w\in \BV(\Omega, \R^d)} \sum_T \alpha_1 \Vert \nabla u_h - w\Vert_{L^1(T)} + \alpha_0 (1+c \varepsilon)\Vert D w\Vert_{\mathcal{M}}+ \sum_T \alpha_1 \Vert w - w_\varepsilon\Vert_{L^1(T)}\\
		&\qquad + \sum_T \alpha_1 \Vert w_\varepsilon - \mathcal{I}_hw_\varepsilon\Vert_{L^1(T)} + \alpha_0\Vert \nabla(w_\varepsilon - \mathcal{I}_h w_\varepsilon)\Vert_{L^1}\\
		&\quad \leq \NSymTGV_{\alpha}^2(u_{h}) + c(\varepsilon + \frac{h}{\varepsilon} + \frac{h^2}{\varepsilon})\NSymTGV_{\alpha}^2(u_{h}).
	\end{align*}
	Using interpolation inequality \eqref{eq:FEinterpolationIneq} and the approximation properties from~\cref{prop:approxBV} for the vector-valued case, we bound all terms except the first two summands in terms of $\Vert Dw\Vert_{\mathcal{M}(\Omega, \R^d)}$.
	The remaining term is a summand inside $\NSymTGV_\alpha^2(u_h)$, yielding the result.
\end{proof}
Using this Lemma, it is now possible to prove fully discrete spatial convergence.
\begin{theorem}
	\label{thm:full_discrete_conv}
	Assume that $\Omega$ is star-shaped. Let $\xi \in L^\infty(\Omega) $ and $\xi_{h} \in \mathcal{V}_h^1$ be defined as in~\cref{def:discTGVdenoiseMin}. Then
	$
	\Vert \xi - \xi_{h}\Vert_{L^2} \leq ch^{1/4}\left(\NSymTGV^2_{\alpha}(\xi) + \Vert \xi\Vert_{L^1}\right)^{1/2}.
	$
\end{theorem}
\begin{proof}
	We do the first step as in~\cref{thm:semi_discrete_conv} and additionally use \cref{rem:BoundDiscTGV} and the optimality of $\xi_{h}$:
	\begin{align*}
		\frac{\lambda}{2}\Vert \xi - \xi_{h}\Vert_{L^2}^2 \leq \psi(\xi_{h}) - \psi(\xi) \leq  \psi_h(\mathcal{I}_h\xi_\varepsilon) - \psi(\mathcal{I}_h\xi_\varepsilon) + \psi(\mathcal{I}_h\xi_\varepsilon) - \psi(\xi).
	\end{align*}
	We already know that $\psi(\mathcal{I}_h\xi_\varepsilon) - \psi(\xi)$ converges as desired, so we investigate the first term.
	Using~\cref{lem:discTGVboundAbove} and repeating calculations from the semidiscrete case gives:
	\begin{align*}
		&\psi_h(\mathcal{I}_h\xi_\varepsilon) - \psi(\mathcal{I}_h\xi_\varepsilon) = \NSymTGV_{h,\alpha}^2(\mathcal{I}_h\xi_\varepsilon) - \NSymTGV_{\alpha}^2(\mathcal{I}_h\xi_\varepsilon)
		\leq c(\varepsilon + \frac{h}{\varepsilon} \\+& \frac{h^2}{\varepsilon})\NSymTGV_{\alpha}^2(\mathcal{I}_h\xi_\varepsilon)
		\leq c (\varepsilon + \frac{h}{\varepsilon} + \frac{h^2}{\varepsilon})(\frac{h}{\varepsilon} + \varepsilon + 1)(\NSymTGV^2_{\alpha}(\xi) + \Vert \xi\Vert_{L^1}).
	\end{align*}
	Combining this calculation with the bounds on $\psi(\mathcal{I}_h\xi_\varepsilon) - \psi(\xi)$ obtained in~\cref{thm:semi_discrete_conv} yields the final bound
	\begin{equation*}
		\Vert \xi - \xi_{h}\Vert_{L^2}^2 \leq c\left( \varepsilon + \varepsilon^2 + h + h^2 + \frac{h}{\varepsilon} + \frac{h^2}{\varepsilon} + \frac{h^2}{\varepsilon^2} +\frac{h^3}{\varepsilon^2}  \right)\left(\NSymTGV^2_{\alpha}(\xi) + \Vert \xi\Vert_{L^1}\right).
	\end{equation*}
	Setting $\varepsilon = h^{1/2}$ and removing higher-order terms concludes the proof.
\end{proof}

\section{Spatiotemporal Total Generalized Variation on Finite Elements}
\label{sec:SpaceTimeReg}.
Some problems like the inverse problem in electrocardiographic imaging are inherently spatiotemporal: an unknown function evolves over time, while the forward model is governed by a PDE posed on a spatial domain.
Motivated by this structure, we extend the spatial framework of the previous section to spatiotemporal $\TGV_\alpha^2$ regularization.
We formulate the spatiotemporal setting, define relevant functionals, establish fundamental properties of the spatiotemporal regularizer, and describe its discretization.
The section concludes with a convergence proof for the fully discrete formulation, extending the explicit convergence rates from the spatial case.

\subsection{Spatiotemporal Function Space}
In the spatiotemporal case, we have a spatial and a temporal domain, $\OmegaSpace$ and $\OmegaTime$, respectively.
The spatial domain  $\OmegaSpace\subset \R^d, d=2, 3$ is again bounded, polygonal, and star-shaped, while the temporal domain is an interval, i.e. $\OmegaTime = (0, \tilde{t})$ for some $\tilde{t} > 0$.
For definiteness, we state results for the symmetric $\TGV$ regularizer; the non-symmetric case follows by the same arguments.
We now define the $\TGV$ for the product domain $\OmegaSpace \times \OmegaTime$ as follows:
\begin{definition}
	\label{def:ContTgv}
	Let $\OmegaSpace \subset \R^d$ and $\OmegaTime = (0, \tilde{t})$ for some $\tilde{t} > 0$.
	Further let $k,l \geq 1$ and $\alpha_0, \ldots, \alpha_{k-1}, \beta_0, \ldots, \beta_{l-1} > 0, \alpha = (\alpha_i)_{i=0}^{k-1}, \beta = (\beta_i)_{i=0}^{l-1}$. Then, the \emph{total generalized variation} of order $(k,l)$ with weights $(\alpha, \beta)$ for $u \in L^1_\text{loc}(\OmegaSpace \times \OmegaTime)$ is defined as the value of the functional
	\begin{equation*}
		\TGV_{(\alpha, \beta)}^{(k,l)}(u) =  \int_{\OmegaTime} \TGV_{\alpha, \mathbf{x}}^k(u) \dx t + \int_{\OmegaSpace} \TGV_{\beta, t}^l(u)\dx\mathbf{x}.
	\end{equation*}
\end{definition}
By defining spatiotemporal $\TGV$ in this integrated form, we later get separate auxiliary variables for each point in space and time, respectively.
This greatly facilitates the numerical implementation compared to directly evaluating $\TGV$ as before for $\Omega = \OmegaSpace\times\OmegaTime$, as we do not need a joint formulation for the derivatives. 
\begin{definition}
	We define the space of spatiotemporal bounded generalized variation as 
	\begin{equation*}
		\BGV_{(\alpha, \beta)}^{(k,l)}(\OmegaSpace \times \OmegaTime) = \left\lbrace u \in  L^1(\OmegaSpace \times \OmegaTime) \colon \TGV_{(\alpha, \beta)}^{(k,l)} < \infty\right\rbrace, 
	\end{equation*}
	with corresponding norm $\Vert  u\Vert _{\BGV_{(\alpha, \beta)}^{(k,l)}} = \Vert u\Vert _{L^1(\OmegaSpace \times \OmegaTime)}+ \TGV_{(\alpha, \beta)}^{(k,l)}$.
\end{definition}
We can again obtain similar properties as for spatial $\TGV$~\cite[Proposition~3.3]{Br10}.
\begin{proposition}
	\label{prop:TGVproperties2}
	The following statements hold:
	\begin{enumerate}[leftmargin=*]
		\item $\TGV_{(\alpha, \beta)}^{(k,l)}(u)$ is a seminorm on the normed space $\BGV_{(\alpha, \beta)}^{(k,l)}(\OmegaSpace \times\OmegaTime)$.
		\item For $u\in L^1_\text{loc}(\OmegaSpace \times \OmegaTime)$, $\TGV_{(\alpha, \beta)}^{(k,l)}(u) = 0$ if and only if $u$ is a polynomial of degree less than k in $\OmegaSpace$ and of degree less than $l$ in $\OmegaTime$.
		\item For fixed $(k,l)$ and positive weights $(\alpha, \beta), (\tilde{{\alpha}}, \tilde{\beta})$ the seminorms $\TGV_{(\alpha, \beta)}^{(k,l)}(u)$ and $\TGV_{(\tilde{\alpha}, \tilde{\beta})}^{(k,l)}(u)$ are equivalent .
		\item $\TGV_{(\alpha, \beta)}^{(k,l)}$ is rotationally invariant; i.e. for each orthonormal matrix $O \in \R^{d \times d}$,  $\tilde{u} \in \BGV_{(\alpha, \beta)}^{(k,l)}\left(O^T \OmegaSpace, \OmegaTime\right)$ and $\TGV_{(\alpha, \beta)}^{(k,l)}(u) = \TGV_{(\alpha, \beta)}^{(k,l)}(\tilde{u})$; $\tilde{u}(x, y) = u\left(O x, y\right)$.
		\item For $(r_1, r_2) > 0$ and $u\in \BGV_{(\alpha, \beta)}^{(k,l)}(\OmegaSpace \times\OmegaTime)$ we have, defining $\tilde{u}(x, y) = u(r_1x, r_2y)$ and $\tilde{\alpha}_i = \alpha_i r_1^{k-i},\tilde{\beta}_j = \beta_j r_2^{l-j}$, that $\tilde{u} \in \BGV_{(\alpha, \beta)}^{(k,l)}(r_1^{-1}\OmegaSpace \times r_2^{-1}\OmegaTime)$ with 
		$
		\TGV_{(\alpha, \beta)}^{(k,l)}(\tilde{u}) = r_1^{-d}r_2^{-1} \left(\int_{\OmegaTime} \TGV_{\tilde{\alpha}, \mathbf{x}}^k(u) \dx t + \int_{\OmegaSpace} \TGV_{\tilde{\beta}, t}^l(u)\dx \mathbf{x}\right). 
		$
	\end{enumerate}
\end{proposition}
\begin{proof}
	Every statement follows from the corresponding statement in~\cite[Proposition 3.3]{Br10} together with basic properties of the integral.
\end{proof}
Due to point 3 we now write $\BGV^{(k,l)}(\OmegaSpace \times\OmegaTime)$ instead of $\BGV_{(\alpha, \beta)}^{(k,l)}(\OmegaSpace \times\OmegaTime)$.
Using similar arguments as in the spatial case~\cite[Proposition~3.5]{Br10}, the resulting space is Banach.
\begin{proposition}
	$\BGV^{(k,l)}(\OmegaSpace \times\OmegaTime)$ is a Banach space when equipped with $\Vert \cdot\Vert _{\BGV_{(\alpha, \beta)}^{(k,l)}}$ for positive weights $(\alpha, \beta)$.
\end{proposition}
\begin{proof}
	We first show that $\TGV_{(\alpha, \beta)}^{(k,l)}$ always gives a lower semi-continuous functional w.r.t. $L^1(\OmegaSpace\times \OmegaTime)$. For that purpose, let the sequence $\lbrace u^n\rbrace$ be in $\BGV^{(k,l)}(\OmegaSpace \times\OmegaTime)$ such that $u^n \rightarrow u$ in $L^1(\OmegaSpace\times \OmegaTime)$. Then for each $v \in \mathcal{C}_c^k(\OmegaSpace, \mathrm{Sym}^k(\R^d))$ with $\Vert \Div^i v\Vert _\infty\leq \alpha_i$ and any $t \in \OmegaTime$ it follows that
	\begin{equation*}
		\int_{\OmegaSpace}u(\cdot, t) \Div^k v \dx \mathbf{x} = \lim_{n\rightarrow\infty} \int_{\OmegaSpace}u^n(\cdot,t) \Div^k v \dx \mathbf{x}  \leq \liminf_{n\rightarrow \infty} \TGV_\alpha^k\left( u^n(\cdot, t) \right).
	\end{equation*}
	Repeating the calculations for the second summand and taking the supremum over $v$, the respective integrals, and using the Lemma of Fatou, then gives the inequality
	\begin{equation*}
		\TGV_{(\alpha, \beta)}^{(k,l)}(u) \leq \liminf_{n\rightarrow\infty} \TGV_{(\alpha, \beta)}^{(k,l)}(u^n).
	\end{equation*}
	This proves that $\TGV_{(\alpha, \beta)}^{(k,l)}$ is lower semi-continuous.
	
	Now, let $\lbrace u^n \rbrace$ be a Cauchy sequence in $\BGV^{(k,l)}(\OmegaSpace \times\OmegaTime)$.
	Standard arguments on the lower semicontinuity then yield
	$
	\TGV_{(\alpha, \beta)}^{(k,l)}(u^n - u) \leq \liminf_{m \rightarrow \infty} \TGV_{(\alpha, \beta)}^{(k,l)}(u^n - u^m) \leq \varepsilon.
	$
	implying that $u^n \rightarrow u$ in $\BGV^{(k,l)}(\OmegaSpace \times\OmegaTime)$.
\end{proof}

To be able to derive convergence statements later, we first prove approximation results by smooth functions akin to~\cref{lem:approxTGV} in the spatiotemporal setting.
\begin{lemma}
	\label{lem:STmollifier}
	Let $\OmegaSpace \subset \R^d, \OmegaTime\subset \R$ be two open, bounded domains that are star-shaped and define for $\mathbf{x} \in {\OmegaSpace}_{\varepsilon}, t \in {\OmegaTime}_{\varepsilon}\quad\tilde{u}_\varepsilon(\mathbf{x}, t) = u(\frac{\mathbf{x}}{1+\varepsilon},\frac{t}{1 + \varepsilon})$.
	Let $\rho_{\varepsilon, 1}, \rho_{\varepsilon, 2}$ be two mollifiers such that~\cref{prop:approxBV} holds for mollified functions on either domain.
	Further let $\rho_\varepsilon(x, y) \coloneqq \rho_{\varepsilon,1}(x)\rho_{\varepsilon,2}(y)$.
	Then  using $\bar{u}_\varepsilon(\mathbf{x}, t) = u(\mathbf{x},\frac{t}{1 + \varepsilon})$ it holds that for any non-negative, convex functional $F$ and $u_\varepsilon = \tilde{u}_\varepsilon*\rho_\varepsilon$
	\begin{equation*}
		\int_{\OmegaSpace}F(u_\varepsilon(\mathbf{x}, \cdot)) \dx \mathbf{x} \leq (1 + c\varepsilon)\int_{\OmegaSpace}F\left( (\bar{u}_\varepsilon(\mathbf{x}, \cdot) * \rho_{\varepsilon,2}\right) \dx \mathbf{x}.
	\end{equation*}
\end{lemma}
\begin{proof}
	A straightforward computation yields
	\begin{align*}
		&\int_{\OmegaSpace}F(u_\varepsilon(\mathbf{x}, \cdot)) \dx \mathbf{x}
		= \int_{\OmegaSpace}F\left(\int_{\R^d}\rho_{\varepsilon,1}(\mathbf{y}) (\tilde{u}_\varepsilon(\mathbf{x} - \mathbf{y}, \cdot) * \rho_{\varepsilon,2} \dx \mathbf{y}\right) \dx \mathbf{x}\\
		&\leq \int_{\OmegaSpace}\int_{\R^d}\rho_{\varepsilon,1}(\mathbf{y})F\left( (\tilde{u}_\varepsilon(\mathbf{x} - \mathbf{y}, \cdot) * \rho_{\varepsilon,2}\right) \dx \mathbf{y} \dx \mathbf{x}
		\\
		&\leq \int_{\R^d}\rho_{\varepsilon,1}(\mathbf{y})\int_{{\OmegaSpace}_{\varepsilon}}F\left( (\tilde{u}_\varepsilon(\mathbf{z}, \cdot) * \rho_{\varepsilon,2}\right) \dx \mathbf{z} \dx \mathbf{y}
		\leq (1 + c\varepsilon)\int_{\OmegaSpace} F\left( (\bar{u}_\varepsilon(\mathbf{x}, \cdot) * \rho_{\varepsilon,2}\right) \dx \mathbf{x},
	\end{align*}
	where we used Jensen's inequality as well as properties of the convolution.
\end{proof}
We will later use this for $F(u) = \TV(u)$ or $\NSymTGV_\alpha^2(u)$, respectively.
The result of this lemma also holds with the same arguments for integrals over $\OmegaTime$.
\subsection{Discretization}
In the spatiotemporal setting, we need to discretize the product domain $\Omega = \OmegaSpace\times  \OmegaTime$.
We use a triangular $\mathcal{P}_1$ discretization $\mathcal{S}_h$ of the time domain $\OmegaTime$, and, as before, a $\mathcal{P}_1$ discretization $\mathcal{T}_h$ of the space domain $\OmegaSpace$.
This then yields prismatic elements on the product space, and we can define the corresponding finite element space as 
\begin{equation}\label{eq:defStTriangulation}
	\mathcal{V}_{\text{ST}, h} \coloneqq \left\lbrace v\in C(\bar{\Omega}) \colon v\lvert_{T \times S}\in \mathcal{P}_1(T) \otimes\mathcal{P}_1(S), \forall (T, S) \in \mathcal{T}_h \times \mathcal{S}_h \right\rbrace.
\end{equation}
Due to the product structure, differentiation of functions in $\mathcal{V}_{\text{ST}, h}$ yields mixed $\mathcal{P}^1$ and $\mathcal{P}^0$ functions.
The gradient space $\mathcal{Q}_{\text{ST}, h}^k$ is hence composed of vector-valued functions $(q_1, \ldots, q_d, q_{d+1})$ piecewise affine in time and constant in space for the first $d$ components representing the spatial derivatives and vice versa for the last component representing the temporal derivative,
\begin{equation}
	\label{eq:defSTSpace}
	\mathcal{Q}_{\text{ST}, h} \coloneqq \left\lbrace q\in L^1(\OmegaSpace\times\OmegaTime)^{d+1}\colon\hspace{-.2cm}\begin{tabular}{l}
		$q_1,\ldots,q_d \lvert_{T \times S}\in \mathcal{P}_0(T)\otimes\mathcal{P}_1(S)$\\
		$q_{d+1}\lvert_{T \times S} \in \mathcal{P}_1(T) \otimes \mathcal{P}_0(S)$
	\end{tabular} \hspace{-.1cm}\forall   (T, S) \in \mathcal{T}_h \times \mathcal{S}_h\right\rbrace. 
\end{equation}
Again, both $\mathcal{V}_{\text{ST}, h} \subset H^1(\OmegaSpace\times {\OmegaTime})$ and $\mathcal{Q}_{\text{ST}, h} \subset L^2(\OmegaSpace\times {\OmegaTime})^{d+1}$ hold.
\subsection{Spatiotemporal Fully Discrete Convergence}
After obtaining all the necessary properties for the spatiotemporal function space and its discretization, we are now able to formulate the denoising objective functional and derive a spatiotemporal convergence theorem.
Again, we investigate the continuous spatiotemporal regularizer
\begin{equation}
	\label{eq:ContIntContReg}
	\NSymTGV_{(\alpha, \beta)}^{(2, 1)}(u) \coloneqq \int_{\OmegaTime} \NSymTGV_{\alpha,\mathbf{x}}^2(u) \dx t+ \int_{\OmegaSpace} \beta \TV_t(u) \dx \mathbf{x},
\end{equation}
and the discrete version
\begin{equation}
	\label{eq:ContIntDiscReg}
	\NSymTGV_{h, (\alpha, \beta)}^{(2, 1)}(u) \coloneqq \int_{\OmegaTime} \NSymTGV_{h,\alpha, \mathbf{x}}^2(u) \dx t + \int_{\OmegaSpace} \beta \TV_t(u) \dx \mathbf{x}.
\end{equation}
These then yield both discrete and continuous spatiotemporal denoising functionals
\begin{equation*}
	\psi(u) = \frac{\lambda}{2}\Vert u - z\Vert_{L^2}^2 +  \NSymTGV_{(\alpha, \beta)}^{(2, 1)}(u), \quad  \psi_h(u) = \frac{\lambda}{2}\Vert u - z\Vert_{L^2}^2 +  \NSymTGV_{h, (\alpha, \beta)}^{(2, 1)}(u).
\end{equation*}
Due to the strict convexity of the continuous functional and the finite-dimensionality of $\mathcal{V}_{\text{ST}, h}^1$, there are unique $\xi \in L^2(\Omega), \xi_{h} \in \mathcal{V}_{\text{ST}, h}^1$ such that 
\begin{equation}
	\label{def:STdiscTGVdenoiseMin}
	\psi(\xi) = \inf \lbrace \psi(u) \rvert u \in L^2(\Omega)\rbrace, \quad\psi_h(\xi_{h}) = \inf \lbrace \psi_h(u) \rvert u \in \mathcal{V}_{\text{ST}, h}^1\rbrace  < \infty.
\end{equation}
By~\cref{lem:STmollifier}, we can apply our spatial convergence results for the spatiotemporal objective functional, introducing only a term of higher order.
This is formalized in the next theorem.
\begin{theorem}
	\label{thm:STfull_discrete_conv}
	Assume that $\OmegaSpace$ and $\OmegaTime$ are star-shaped. Let $\xi \in L^\infty(\Omega)$ and $\xi_{h}\in\mathcal{V}_{\text{ST}, h}^1$ be defined as in~\cref{def:STdiscTGVdenoiseMin}. Then
	$
	\Vert \xi - \xi_{h}\Vert_{L^2} \leq ch^{1/4}\left(\Vert \xi \Vert_{L^1} + \NSymTGV_{(\alpha, \beta)}^{(2, 1)}(\xi)\right)^{1/2}.
	$
\end{theorem}
\begin{proof}
	Let $\varepsilon > 0$ and $\xi_\varepsilon$ be as in~\cref{lem:STmollifier}. Again by the strong convexity of $\psi$, the optimality of $\xi$ and $\xi_{h}$ and pointwise application of~\cref{rem:BoundDiscTGV} we estimate
	\begin{align*}
		&\frac{\lambda}{2}\Vert \xi - \xi_{h}\Vert_{L^2}^2 \leq \psi(\xi_{h}) - \psi(\xi) \leq \psi_h(\mathcal{I}_h\xi_\varepsilon) - \psi(\xi)= \frac{\lambda}{2}(\Vert \mathcal{I}_h\xi_\varepsilon - z\Vert_{L^2}^2 -  \Vert \xi - z\Vert_{L^2}^2)\\
		&+\NSymTGV_{h, (\alpha, \beta)}^{(2, 1)}(\mathcal{I}_h\xi_\varepsilon) -\NSymTGV_{(\alpha, \beta)}^{(2, 1)}(\xi)   = \mathcal{A}_1 + \mathcal{A}_2.
	\end{align*}
	$\mathcal{A}_1$:
	Using the approximation properties for the Clément interpolation~\cref{eq:FEinterpolationIneq}, \cref{prop:approxBV} and the essential boundedness of $\xi$ and $z$ yields
	\begin{equation*}
		\Vert \mathcal{I}_h\xi_\varepsilon - z\Vert _{L^2}^2 -  \Vert \xi - z\Vert _{L^2}^2 \leq c\left(\frac{h^2}{\varepsilon} + \varepsilon\right)\left(\Vert \xi \Vert_{L^1} + \NSymTGV_{(\alpha, \beta)}^{(2, 1)}(\xi)\right).
	\end{equation*}
	$\mathcal{A}_2$:
	Due to the convexity of the spatiotemporal total generalized variation, we can bound
	\begin{align*}
		&\NSymTGV_{h, (\alpha, \beta)}^{(2, 1)}(\mathcal{I}_h\xi_\varepsilon) -\NSymTGV_{(\alpha, \beta)}^{(2, 1)}(\xi)\\
		&\leq \NSymTGV_{h, (\alpha, \beta)}^{(2, 1)}(\mathcal{I}_h\xi_\varepsilon) - \NSymTGV_{(\alpha, \beta)}^{(2, 1)}(\mathcal{I}_h\xi_\varepsilon)\\
		&\qquad+\NSymTGV_{(\alpha, \beta)}^{(2, 1)}(\mathcal{I}_h\xi_\varepsilon -\xi_\varepsilon)  + \NSymTGV_{(\alpha, \beta)}^{(2, 1)}(\xi_\varepsilon)-\NSymTGV_{(\alpha, \beta)}^{(2, 1)}(\xi)).
	\end{align*}
	The integrals of $\TV_t(\mathcal{I}_h\xi_\varepsilon)$ cancel each other out in the first two terms.
	Repeating calculations from the proof of \cref{thm:full_discrete_conv} inside the other integral yields
	\begin{align*}
		&\NSymTGV_{h, (\alpha, \beta)}^{(2, 1)}(\mathcal{I}_h\xi_\varepsilon) - \NSymTGV_{(\alpha, \beta)}^{(2, 1)}(\mathcal{I}_h\xi_\varepsilon)\\
		&\qquad\leq c(\varepsilon + \varepsilon^2 + h + h^2 + \frac{h}{\varepsilon} + \frac{h^2}{\varepsilon} + \frac{h^2}{\varepsilon^2} + \frac{h^3}{\varepsilon^2} )\left(\Vert \xi \Vert_{L^1} + \NSymTGV_{(\alpha, \beta)}^{(2, 1)}(\xi)\right)
	\end{align*}
	For the third and fourth terms,  we use~\cref{lem:STmollifier} and repeat calculations from the proofs of~\cref{thm:semi_discrete_conv} and~\cite[Theorem~7.1]{Ba14}.
	For the third term, this yields the bound
	\begin{equation*}
		\NSymTGV_{h, (\alpha, \beta)}^{(2, 1)}(\mathcal{I}_h\xi_\varepsilon -\xi_\varepsilon) \leq c (h + \frac{h}{\varepsilon})\left(\Vert \xi \Vert_{L^1} + \NSymTGV_{(\alpha, \beta)}^{(2, 1)}(\xi)\right),
	\end{equation*}
	while for the fourth term, we compute
	\begin{align*}
		&\NSymTGV_{(\alpha, \beta)}^{(2, 1)}(\xi_\varepsilon) \leq(1 + c(\varepsilon + \varepsilon^2)) \NSymTGV_{(\alpha, \beta)}^{(2, 1)}(\xi).
	\end{align*}
	The summand with factor 1 then cancels out with the fifth term.
	Setting $\varepsilon = h^{1/2}$, adding $\Vert \xi \Vert_{L^1}$ to the last term, and removing higher-order terms concludes the proof.
\end{proof}

\section{Numerical Results}
\label{sec:NumRes}
We begin by deriving a discrete formulation of the regularizer $\NSymTGV_{h, (\alpha,\beta)}^{2,1}(u)$ in \eqref{eq:ContIntDiscReg}.
Building on this formulation, we introduce a first-order primal-dual optimization framework~\cite{Ch11} and compute the proximal operators associated with $\NSymTGV_{h, (\alpha,\beta)}^2(u)$ explicitly.
The numerical section is organized around two complementary objectives.
First, we consider canonical imaging problems, namely denoising and inpainting, which serve as controlled benchmark settings for experimentally validating the convergence rates from \cref{thm:full_discrete_conv} and \cref{thm:STfull_discrete_conv}.
Second, and as the main application-oriented component of this work, we apply the proposed spatiotemporal regularization framework to PDE-constrained inverse problems, focusing on the inverse problem in electrocardiographic imaging.
For ECGI, we provide a comprehensive introduction to the forward problem, prove the existence of minimizers for the considered variational formulations, and compare against standard regularization approaches from the literature.
The resulting reconstructions are assessed both quantitatively and qualitatively.
Throughout this section, boldface functionals indicate their application to finite element functions rather than to continuous functions.

\subsection{Numerical Implementation of Finite Element TGV}
The spatiotemporal discretization $\mathcal{V}_{\mathrm{ST},h}^1$~\eqref{eq:defStTriangulation} induces functions of the form
\begin{equation}
	\label{eq:SpaceTimeDiscFunction}
	u_h(x,t) = \sum_{i,j=1}^{N,M}a_{ij}\phi_i(x)\hat{\phi}_j(t).
\end{equation}
Here, $ N$ and $ M$ are the respective numbers of spatial and temporal DOFs.
The basis functions $\phi, \hat{\phi}$ are as defined in \eqref{eq:defStTriangulation}, and the $a_{ij}$ represent the nodal values of the function.
Since the auxiliary variable $w$ differs at every timestep, we can keep its purely spatial discretization $\mathcal{V}_h^2$. 
Using the fact that both the spatial and temporal regularizers inside the respective integrals in \eqref{eq:ContIntDiscReg} are convex and lumping the quadrature weights, we approximate as follows:
\begin{equation}
	\label{eq:usedTGVAlgo}
	\begin{aligned}
		&\int_{\OmegaTime} \NSymTGV_{h,\alpha, \mathbf{x}}^2(u_h) \dx t + \int_{\OmegaSpace} \beta \TV_t(u_h) \dx \mathbf{x}\\
		\approx& \sum_{j=1}^M\NSymTGV_{h,\alpha}^2\left(\sum_{i=1}^Na_{ij}\phi_i\right) + \sum_{i=1}^N\beta \TV\left(\sum_{j=1}^M a_{ij}\hat{\phi}_j\right) \\
		=&\sum_{j=1}^M\min_{w_j\in\mathcal{V}_h^2}\alpha_1 \vert \nabla_x u_h(\cdot, t_j) - w_j\vert_\gamma +\alpha_0\vert \nabla w_j\vert_\gamma +  \sum_{i=1}^N\beta \vert\nabla_t u_h(x_i, \cdot)\vert\eqqcolon\mathbf{TGV_{P1}^\gamma}(u_h),
	\end{aligned}
\end{equation}
where the parameter $\gamma=1, 2$ determines the norm used for the spatial gradients, corresponding to anisotropic and isotropic variants \cite{Bi11}.
While this approximation avoids any quadrature rules on the integrals, it is sufficient to observe a convergence rate of magnitude $h^{1/4}$ in numerical experiments.

\subsection{Primal-Dual Optimization Algorithm}

Given a noisy input $z_h$, we consider data fidelity terms of the form $\mathbf{G}(u_h) = \frac{1}{2}\Vert Au_h - z_h\Vert_2^2$ throughout the numerical experiments.
Adding a regularizer $\mathbf{F}$ such as $\mathbf{TGV_{P1}^\gamma}(u_h)$ and minimizing the resulting functional then yields a regularized solution given $z_h$.
To compute such a closed-form solution for regularized minima is, in general, not possible.
However, we consider both $\mathbf{F}$ and $\mathbf{G}$ to be convex and can hence leverage powerful tools from convex analysis to compute regularized minima.
In this work, we use the primal-dual algorithm~\cite{Ch11}, as it is provably optimal with convergence guarantees.

In a general setting, the primal-dual algorithm minimizes an objective functional $\mathbf{E}(u_h) = \mathbf{G}(u_h) + \mathbf{F}(Ku_h)$.
Both functionals map from Hilbert spaces to the nonnegative real numbers, and $K$ is a linear operator between these Hilbert spaces.
In the case of $\mathbf{TGV_{P1}^\gamma}(u_h)$, $K$ consists of
\begin{equation*}
	(\nabla_x, \nabla_t):\mathcal{V}_{\text{ST}, h}^1 \to \mathcal{Q}_{\text{ST}, h} 
	\quad u_h\mapsto (p_h, r_h),\qquad\nabla: \mathcal{V}_h^2 \to \mathcal{V}_h^{4} \quad
	w_h^j \mapsto q_h^j, \; j=1,\ldots, M,
\end{equation*}
where $(p_h, q_h^j, r_h)$ are the so-called dual variables.
For ease of notation, $w_h$ will from here on denote the vector of all auxiliary variables $w^j_h$, while $q_h$ will denote the same for the corresponding dual variables $q^j_h$ for all $j=1, \ldots, M$.
The gradient is applied component-wise to $w_h$.
The dual operators $\Div_x, \Div_t$ and $\Div$ are defined such that $(\nabla_xu_h, p_h) = (u_h, \Div_x p_h), (\nabla_tu_h, r_h) = (u_h, \Div_t r_h)$ and $(\nabla w_h, q_h) = (w_h, \Div q_h)$, where all scalar products are taken on the respective gradient spaces.
Given a data fidelity term $\mathbf{G}(u_h)$, we can rewrite the optimization problem in terms of both primal variables as $\min_{u_h,w_h}\mathbf{G}(u_h) +  \mathbf{TGV_{P1}^2}(u_h,w_h)$.
Now we rewrite it in the saddle point form,
\begin{align*}
	\min_{u_h,w_h} \max_{p_h, q_h, r_h} \mathbf{G}(u_h) &+ (\nabla_x u_h - w_h, p_h) + (\nabla w_h, q_h) + (\nabla_tu_h, r_h)\\ &- (\alpha_1\vert p_h\vert_2)^* - (\alpha_0\vert q_h\vert_2)^*- (\beta\vert r_h\vert_2)^*.
\end{align*}
As a final step, we need to compute proximal maps of $G$ as well as the convex conjugates of the norms and their respective proximal maps.
The proximal map of $G$ is assumed to be simple and will later be computed separately for each application.
A straightforward computation yields $\prox_{\sigma (\alpha_1\vert\cdot\vert_2)^*}(p_h) =\proj_{C_p}(p_h) \quad C_p=\lbrace p_h \colon \vert p_{h} \vert_2 \leq \alpha_1\rbrace$, and equivalent projections for variables $q_h$ and $r_h$.
We can now define~\cref{alg:algorithm} which is known to converge provided the stepsizes $\tau, \sigma > 0$ are chosen such that $\tau\sigma\lambda_\mathrm{max}(K^*K)\leq 1$ where
\begin{equation*}
	K= \begin{pmatrix}
		\nabla_x&-I\\
		0&\nabla\\
		\nabla_t& 0
	\end{pmatrix}\quad K^*= \begin{pmatrix}
		\Div_x& 0 &\Div_t\\
		-I&\Div&0\\
	\end{pmatrix}.
\end{equation*}
\begin{algorithm}
	\caption{First-Order Primal-Dual Algorithm for $ \mathbf{TGV_{P1}^2}$}\label{alg:algorithm}
	\begin{enumerate}
		\item \textbf{Initialization:} Choose $\tau, \sigma > 0$ with $\tau\sigma\lambda_\mathrm{max}(K^*K)\leq 1$, $\theta \in [0, 1]$, $u^0= \bar{u}^0=w^0=\bar{w}^0=p^0=q^0=r^0=0$,
		\item \textbf{Iterations} ($n \geq 0$): Update $u^n, w^n, \bar{u}^n$ as follows:
		\[
		\begin{cases}
			p^{n+1} = \proj_{C_p}(p^n + \sigma (\nabla_x\bar{u}^n - w^n)) \\
			q^{n+1} = \proj_{C_q}(q^n + \sigma \nabla \bar{w}^n) \\
			r^{n+1} = \proj_{C_r}(r^n + \sigma \nabla_t\bar{u}^n) \\
			u^{n+1} = \prox_{\tau \mathbf{G}}(u^n - \tau (\Div_x(p^{n+1}) +\Div_t(r^{n+1}))) \\
			w^{n+1} = w^n + \tau(p^{n+1}- \Div(q^{n+1}))\\
			\bar{u}^{n+1} = u^{n+1} + \theta(u^{n+1} - u^n)\\
			\bar{w}^{n+1} = w^{n+1} + \theta(w^{n+1} - w^n)
		\end{cases}
		\]
	\end{enumerate}
\end{algorithm}

This algorithm can now easily be adapted to the other regularizers.
We can switch to $\gamma=1$ by changing $C_p$ to $\lbrace p_h: \vert p_h\vert_\infty\leq \alpha_1\rbrace$ and optimize purely spatial problems by removing the temporal gradient and $r_h$.
Further, we can optimize total variation regularized problems by removing $w_h$ and all associated operators.
To speed up practical convergence, we used diagonal preconditioning~\cite{Po11} over the different gradient and divergence operators.
\subsection{Image Reconstruction Results}
In this section, we present qualitative and quantitative results for two classical inverse problems in mathematical imaging, inpainting and denoising in the spatial setting.
After computing proximal maps for both image denoising and inpainting, and introducing baseline methods, we compare reconstruction results both qualitatively and quantitatively on the BSD68 dataset~\cite{Ma01}.

In all image reconstruction experiments, additive white noise $\mathbf{n}\sim \mathcal{N}(0,\sigma^2I)$ was added with a standard deviation $\sigma$ of $5\%$ of the data range.
To compare methods, we choose regularization parameters independently and optimally for each approach.
In all experiments, the parameters were optimized w.r.t. the $L^2$ mismatch between reconstructed solutions and original images, stopping each optimization algorithm when the $\ell^\infty$-distance between successive iterates reaches $10^{-3}$.
For the simpler gradient image, we instead iterated until $10^{-5}$.
For methods using piecewise linear finite elements, each nodal degree of freedom (DOF) represents a pixel value, while for piecewise constant finite elements, each pixel is partitioned into two triangles, as done in~\cite{Ba23}.
Both triangle values are subsequently averaged before computing the error on the resulting pixel grid.
We utilize the average peak signal-to-noise ratio (PSNR) defined as $\mathrm{PSNR}(u, u^g) = 10 \log_{10}(\mathrm{MSE}(u, u^g)^{-1})$, where $u^g$ and $u$ are the ground truth and reconstructed images, respectively.
Additionally, we use the perception-based structural similarity index (SSIM)~\cite{Zh04} metric to compare single images qualitatively.
Data fidelity terms $\mathbf{G}(u_h)$ and corresponding proximal maps can be found in~\cref{tab:proxmapsImaging}.
\begin{table}
	\centering
	\renewcommand{\arraystretch}{1.45}
	\setlength{\tabcolsep}{8pt}
	\begin{tabular}{p{0.22\textwidth}|p{0.3\textwidth}|p{0.32\textwidth}}
		\centering\textbf{Example} 
		& \centering\textbf{Data fidelity term} 
		& \centering\arraybackslash\textbf{Proximal map }$\boldsymbol{\prox_{\tau \mathbf{G}}(u_h)}$ \\
		\hline
		
		Image denoising
		& $\displaystyle \frac{1}{2}\Vert u_h-z_h\Vert_{\mathcal{V}_h}^2$
		& $\displaystyle \frac{\tau z_h + u_h}{1+\tau}$ \\
		\hline
		
		Image inpainting
		& $\displaystyle \frac{1}{2}\Vert u_h-z_h\Vert^2_{\mathcal{V}_h(\Omega\setminus I)}$
		& $\displaystyle 
		\frac{\tau z_h + u_h}{1+\tau}\,\mathbf{1}_{\Omega\setminus I}
		+ u_h\,\mathbf{1}_{I}$ \\
	\end{tabular}
	\caption{Data fidelity terms and corresponding proximal maps for two examples in mathematical imaging. Here, $I\subset\Omega$ denotes the set of missing pixels.}
	\label{tab:proxmapsImaging}
\end{table}

A natural comparison method is first-order total variation for piecewise linear functions on grids~\cite{Ba14}.
This can be defined in the discrete setting for $u_h \in \mathcal{V}_h^1$ as
\begin{equation*}
	\mathbf{TV_{P1}}(u_h) \coloneqq \alpha \sum_{T\in\mathcal{T}_h} \vert\nabla u_h\vert_2.
\end{equation*}
For optimization, we can reuse the primal-dual algorithm as outlined in the preceding section.
Apart from piecewise linear finite elements, it is also possible to discretize $\NSymTGV_\alpha^2$ for grids with piecewise constant finite elements~\cite{Ba23}.
For discretizing the auxiliary function $w$, the Raviart-Thomas finite element space with a zero boundary was chosen.
Due to the continuity of the normal over the edges, the degrees of freedom for $w \in \mathcal{RT}_0(\Omega)$ are defined on each edge $E$ as $\int_E w\cdot\nu \dx S = \vert E\vert w_+\cdot\nu_+$.
Note that this exactly corresponds to the location of the jumps of the piecewise constant finite element function, as both possess one value per interior and no values on each boundary facet of the domain.
The discretized $\NSymTGV_\alpha^2$-functional given a piecewise constant $u \in \mathcal{Q}_h^1$~\eqref{eq:pw_constFEspace} is defined as
\begin{align*}
	\mathbf{TGV_{P0}}(u_h) \coloneqq&\min_{w\in \mathcal{RT}_0(\Omega)} \alpha_1 \sum_E \vert E \vert \left\vert \llbracket u_h \rrbracket + h_E w_{h+}\cdot \nu_+\right\vert \\&+ \alpha_0 \sum_T \vert T\vert\vert\nabla w_h\vert_F + \alpha_0 \sum_E \sum_{i=1}^2 \frac{E}{2} \vert \llbracket w_h\rrbracket (X_{E,i)}\vert_2 .
\end{align*}

The crucial part in this formulation is the factor $h_E \coloneqq \vert m_+ - m_-\vert_2$, the distance of the circumcenters $m_+, m_-$ of two neighboring triangles.
By using this factor, it is possible to see that the kernel of the $\mathbf{TGV_{P0}}$-functional in 2D exactly consists of piecewise constant functions interpolating linear functions at the circumcenters of each triangle~\cite[Section 3.2]{Ba23}.
For optimization, we employ the split Bregman algorithm~\cite{Go09}; a more detailed description of both the regularization and the optimization can be found in the original paper.

As a first experiment in image denoising, we present reconstructions of the widely used inverted gradient image; see e.g.~\cite{Ba23, Se11}.
The image of size 128 by 128 pixels features both sharp edges and linear transitions; qualitative results are shown in~\cref{fig:1main}.
In the reconstructions we see that $\mathbf{TGV_{P0}}$, $\mathbf{TGV^1_{P1}}$ and $\mathbf{TGV^2_{P1}}$ clearly outperform $\mathbf{TV_{P1}}$, which displays very visible staircasing artifacts.
Visually, $\mathbf{TGV_{P0}}$ and $\mathbf{TGV_{P1}^1}$ produce nearly indistinguishable reconstructions, despite $\mathbf{TGV_{P1}^1}$ using half the number of degrees of freedom in the computation.
In terms of SSIM on the inverted gradient image, $\mathbf{TGV_{P1}^2}$ performs marginally worse than the other second-order methods.
Additionally, we compare all methods quantitatively on the BSD68 database~\cite{Ma01} in~\cref{tab:BSD68comp}.
Here, all methods achieve similar PSNR values, with $\mathbf{TGV^2_{P1}}$ attaining the highest score.

\begin{figure}[htbp]
	\centering
	
	\begin{subfigure}[t]{0.3\textwidth}
		\centering
		\includegraphics[width=\textwidth]{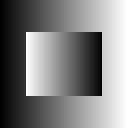}
		\caption{Ground truth image}
		\label{fig:1a}
	\end{subfigure}
	\hfill
	\begin{subfigure}[t]{0.3\textwidth}
		\centering
		\includegraphics[width=\textwidth]{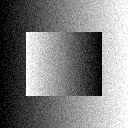}
		\caption{Noisy image}
		\label{fig:1b}
	\end{subfigure}
	\hfill
	\begin{subfigure}[t]{0.3\textwidth}
		\centering
		\includegraphics[width=\textwidth]{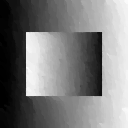}
		\caption{$\mathbf{TV_{P1}}$ reconstruction, SSIM=0.924, $\alpha_0=5.51\cdot10^{-2}$}
		\label{fig:1c}
	\end{subfigure}
	
	\begin{subfigure}[t]{0.3\textwidth}
		\centering
		\includegraphics[width=\textwidth]{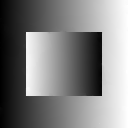}
		\caption{$\mathbf{TGV_{P0}}$ reconstruction, SSIM=$\mathbf{0.998}$, $\alpha_0=8.42\cdot10^{-2}$, $\alpha_1=4.93\cdot10^{-2}$}
		\label{fig:1d}
	\end{subfigure}
	\hfill
	\begin{subfigure}[t]{0.3\textwidth}
		\centering
		\includegraphics[width=\textwidth]{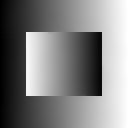}
		\caption{$\mathbf{TGV_{P1}^1}$ reconstruction, SSIM=$\mathbf{0.998}$, $\alpha_0=7.63\cdot10^{-2}$, $\alpha_1=4.09\cdot10^{-2}$}
		\label{fig:1e}
	\end{subfigure}
	\hfill
	\begin{subfigure}[t]{0.3\textwidth}
		\centering
		\includegraphics[width=\textwidth]{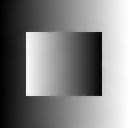}
		\caption{$\mathbf{TGV_{P1}^2}$ reconstruction, SSIM=0.996, $\alpha_0=6.80\cdot10^{-2}$, $\alpha_1=4.55\cdot10^{-2}$}
		\label{fig:1f}
	\end{subfigure}
	
	\caption{Qualitative comparison of reconstructions of the inverted gradient image}
	\label{fig:1main}
\end{figure}

For the image inpainting example, a text mask occluding 5\% of the pixels was used.
As a qualitative example, we present reconstruction results on a sample image of the BSD68 database in~\cref{fig:inpainting_res}.
This image was selected such that the performance difference between $\mathbf{TGV_{P0}}$ and our best-performing method, $\mathbf{TGV_{P1}^2}$, is the median among all BSD68 images.
There, we can observe that the filled-in mask for $\mathbf{TV_{P1}}$ is piecewise constant and clearly visible after reconstruction.
$\mathbf{TGV_{P0}}$, $\mathbf{TGV^1_{P1}}$ and $\mathbf{TGV^2_{P1}}$ again have very similar results. 
We highlight the prominently visible diagonal beams in the picture, which are reconstructed even in the occluded regions of the image.
Quantitative results can be found in~\cref{tab:BSD68comp}.
As in the denoising example, $\mathbf{TGV^2_{P1}}$ achieves the best PSNR value on the entire BSD68 dataset.
\begin{figure}[htbp]
	\centering
	
	\begin{subfigure}[t]{0.49\textwidth}
		\centering
		\includegraphics[angle=270, width=\textwidth]{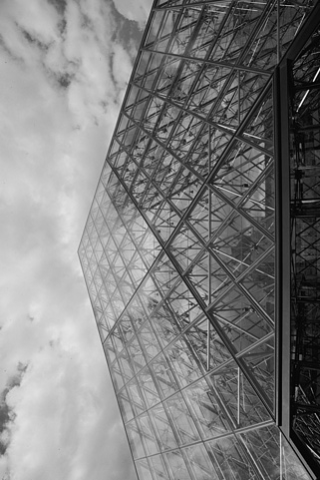}
		\caption{Ground truth image}
		\label{fig:Inpainting_a}
	\end{subfigure}
	\hfill
	\begin{subfigure}[t]{0.49\textwidth}
		\centering
		\includegraphics[angle=270, width=\textwidth]{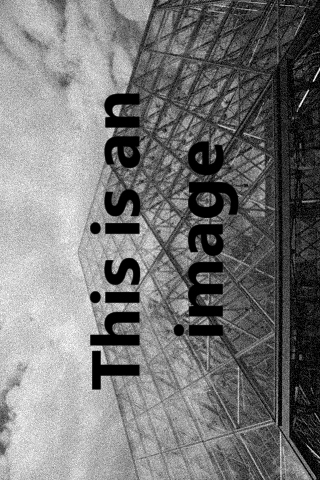}
		\caption{Noisy image with inpainting mask}
		\label{fig:Inpainting_b}
	\end{subfigure}
	
	\begin{subfigure}[t]{0.49\textwidth}
		\centering
		\includegraphics[angle=270, width=\textwidth]{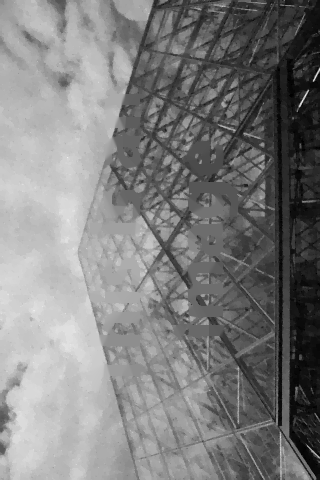}
		\caption{$\mathbf{TV_{P1}}$ reconstruction, SSIM=0.857}
		\label{fig:Inpainting_c}
	\end{subfigure}
	\hfill
	\begin{subfigure}[t]{0.49\textwidth}
		\centering
		\includegraphics[angle=270, width=\textwidth]{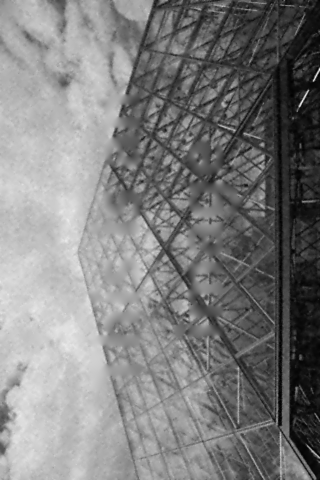}
		\caption{$\mathbf{TGV_{P0}}$ reconstruction, SSIM=0.869}
		\label{fig:Inpainting_d}
	\end{subfigure}
	
	\begin{subfigure}[t]{0.49\textwidth}
		\centering
		\includegraphics[angle=270, width=\textwidth]{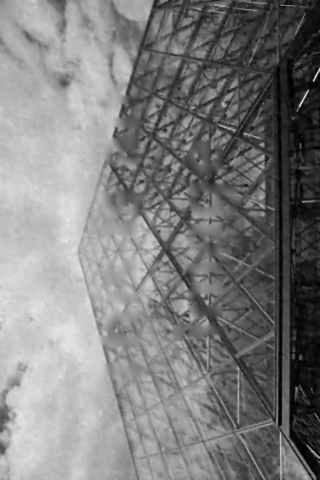}
		\caption{$\mathbf{TGV_{P1}^1}$ reconstruction, SSIM=0.871}
		\label{fig:Inpainting_e}
	\end{subfigure}
	\hfill
	\begin{subfigure}[t]{0.49\textwidth}
		\centering
		\includegraphics[angle=270, width=\textwidth]{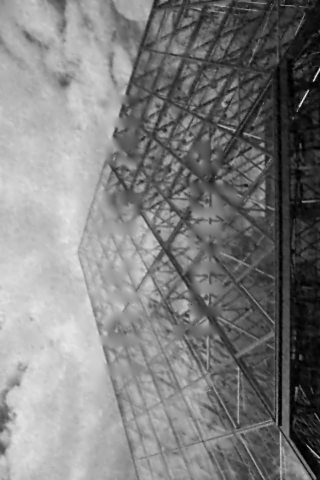}
		\caption{$\mathbf{TGV_{P1}^2}$ reconstruction, SSIM=0.871}
		\label{fig:Inpainting_f}
	\end{subfigure}
	
	\caption{Qualitative comparison of inpainting reconstructions on image 43 of the BSD68 dataset, parameters as in \cref{tab:BSD68comp}}
	\label{fig:inpainting_res}
\end{figure}
\begin{table}[htbp]
	\centering
	\resizebox{\textwidth}{!}{\begin{tabular}{c|c|c|c|c}
			\diagbox{Prob.}{Meth.} 
			& $\mathbf{TV_{P1}}$ 
			& $\mathbf{TGV_{P0}}$ 
			& $\mathbf{TGV^1_{P1}}$
			& $\mathbf{TGV^2_{P1}}$ \\
			\hline
			Denoising 
			& \begin{tabular}{c}
				PSNR: 31.849\\
				$\alpha_0 = 3.23 \cdot 10^{-2}$
			\end{tabular}
			& \begin{tabular}{c}
				PSNR: 31.851\\
				$\alpha_0 = 8.97 \cdot10^{-3}$ \\
				$\alpha_1 = 3.03\cdot10^{-2}$
			\end{tabular}
			& \begin{tabular}{c}
				PSNR: 32.050\\
				$\alpha_0 = 6.61 \cdot10^{-3} $ \\
				$\alpha_1 = 3.36 \cdot10^{-2} $
			\end{tabular}
			& \begin{tabular}{c}
				PSNR: $\mathbf{32.090}$\\
				$\alpha_0 = 7.76 \cdot10^{-3} $ \\
				$\alpha_1 = 3.30 \cdot10^{-2} $
			\end{tabular}
			\\
			\hline
			Inpainting 
			& \begin{tabular}{c}
				PSNR: 29.253\\
				$\alpha_0 = 3.42\cdot 10^{-2}$
			\end{tabular}
			& \begin{tabular}{c}
				PSNR: 29.815\\
				$\alpha_0 = 7.81\cdot10^{-3}$ \\
				$\alpha_1 = 1.10\cdot10^{-1}$
			\end{tabular}
			& \begin{tabular}{c}
				PSNR: 29.871\\
				$\alpha_0 = 7.14\cdot10^{-3}$ \\
				$\alpha_1 = 3.36\cdot10^{-2}$
			\end{tabular}
			& \begin{tabular}{c}
				PSNR: $\mathbf{29.874}$\\
				$\alpha_0 = 9.13\cdot10^{-3}$ \\
				$\alpha_1 = 4.05\cdot10^{-2}$
			\end{tabular}
			\\
		\end{tabular}
	}
	\caption{Quantitative comparison of methods for denoising and inpainting over the BSD68 dataset}
	\label{tab:BSD68comp}
\end{table}

For each finite element discretization in the experiments, the FEniCSx library~\cite{Bar23} was used in conjunction with Basix~\cite{Sc22, Scr22}.
The split Bregman algorithm used to optimize $\mathbf{TGV_{P0}}$ uses a system matrix that was derived with the help of the unified form language UFL~\cite{Al14}.
All computations are performed using CuPy~\cite{CUPY} and SciPy~\cite{SCIPY} on an Intel Xeon Platinum 8268 CPU and an NVIDIA TITAN RTX GPU in double precision.
The optimal parameters for each method were obtained by hyperparameter optimization using the Optuna framework~\cite{Ak19}.

\subsection{Convergence results}
To experimentally investigate the convergence results claimed in~\cref{thm:full_discrete_conv} and~\cref{thm:STfull_discrete_conv}, we utilized both regular and unstructured meshes in the spatial and spatiotemporal settings for $\mathbf{TGV^2_{P1}}$.
For the spatial setting, we restored the inverted gradient image from~\cref{fig:1a} on meshes with resolutions of 4 by 4 to the original 128 by 128 pixels by uniformly refining the underlying mesh.
As the computation of continuous, regularized minima is notoriously difficult, we used the regularized solution, $\xi_N$, on the highest resolution image grid as a surrogate for the continuous optimum.
Additionally, we used an irregular mesh representing a circle with radius 1 and similar element sizes to the grid.
The coarse circle mesh was created via gmsh~\cite{Ge09} with mesh size $h=0.35$ and then uniformly refined five times.
On the circle mesh, we reconstruct a similar ground truth function featuring both linear transitions as well as sharp edges. 
Plots of both ground truth functions are presented in \cref{fig:snapshot_convergence}.
To represent both problems in a spatiotemporal setting, the central region in each example was scaled to take up 0 to 100\% of the image over time.
The parameters $\alpha = (\alpha_0, \alpha_1)$ were chosen as the optimal parameters in the gradient image reconstruction, while $\beta$ was comparatively chosen as $4\cdot10^{-2}$.
Results can be found in~\cref{fig:hconvfig}.

We are able to observe a convergence rate of $h^{1/2}$ in both spatial and a rate of  $h$ in both spatiotemporal experiments.
\begin{figure}[htp]
	\centering
	\hfill\includegraphics[width=0.35\linewidth]{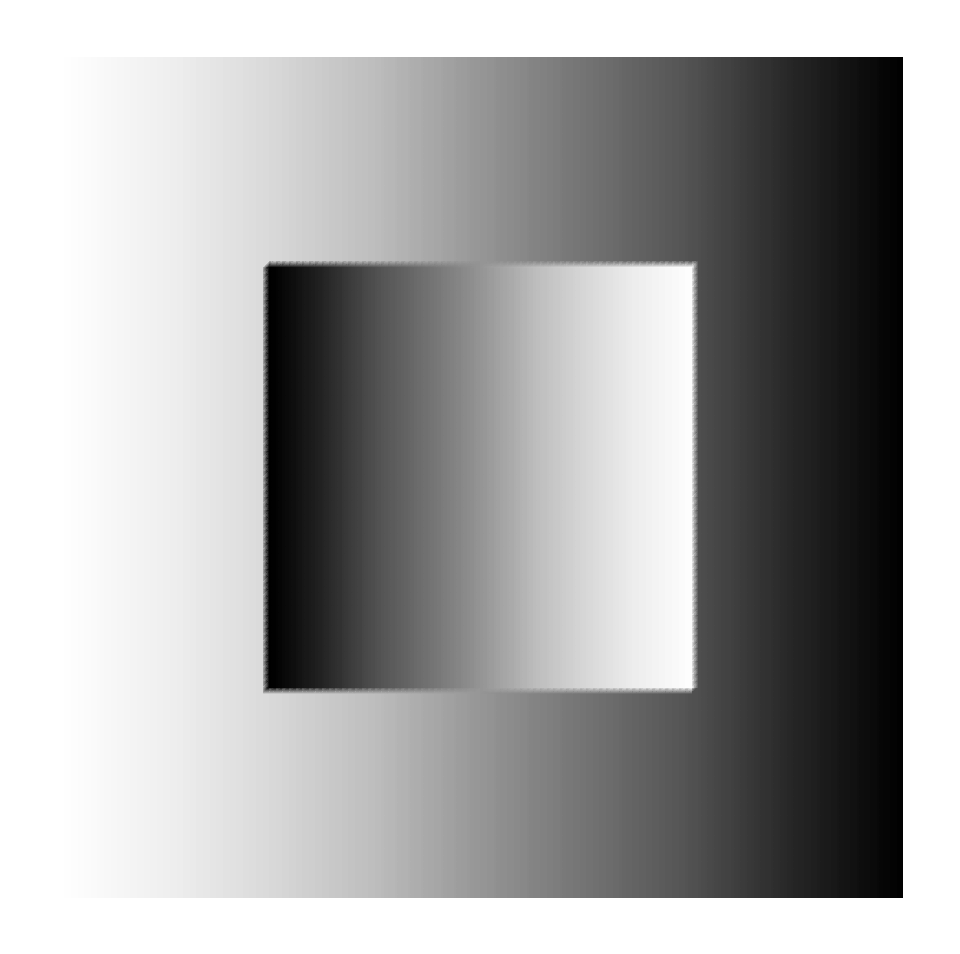}\hfill
	\includegraphics[width=0.35\linewidth]{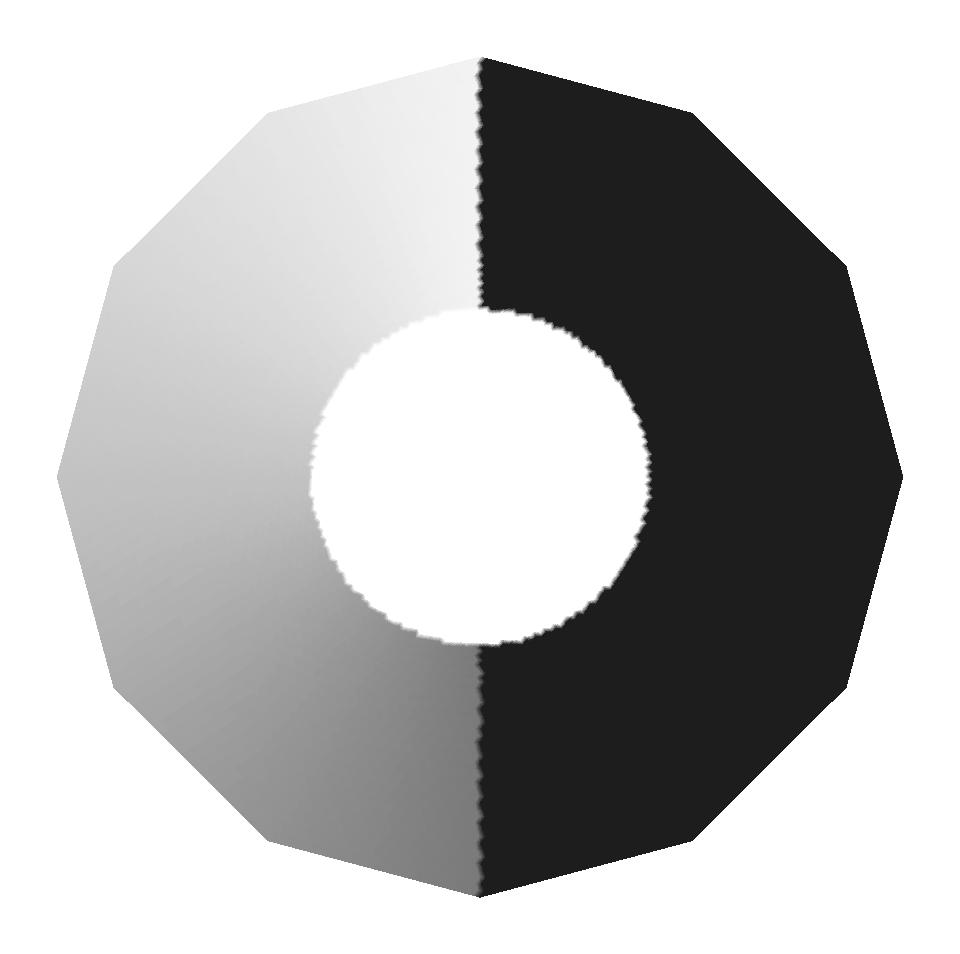}\hfill
	\hspace*{.3cm}
	\caption{Ground truth functions used in spatial convergence on the finest grid, correspond to ground truth function at middle timestep in spatiotemporal convergence}
	\label{fig:snapshot_convergence}
\end{figure}
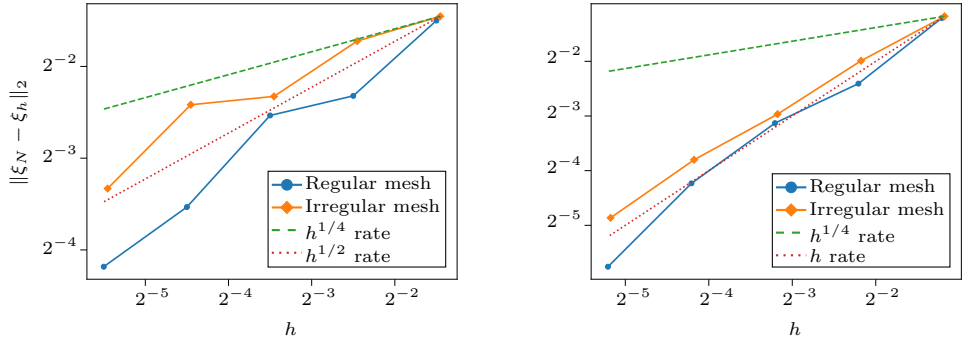
\begin{figure}
	\centering
	\begin{tikzpicture}[font=\scriptsize]
		
		\node (img) at (0,0)
		{\includegraphics[width=0.465\textwidth]{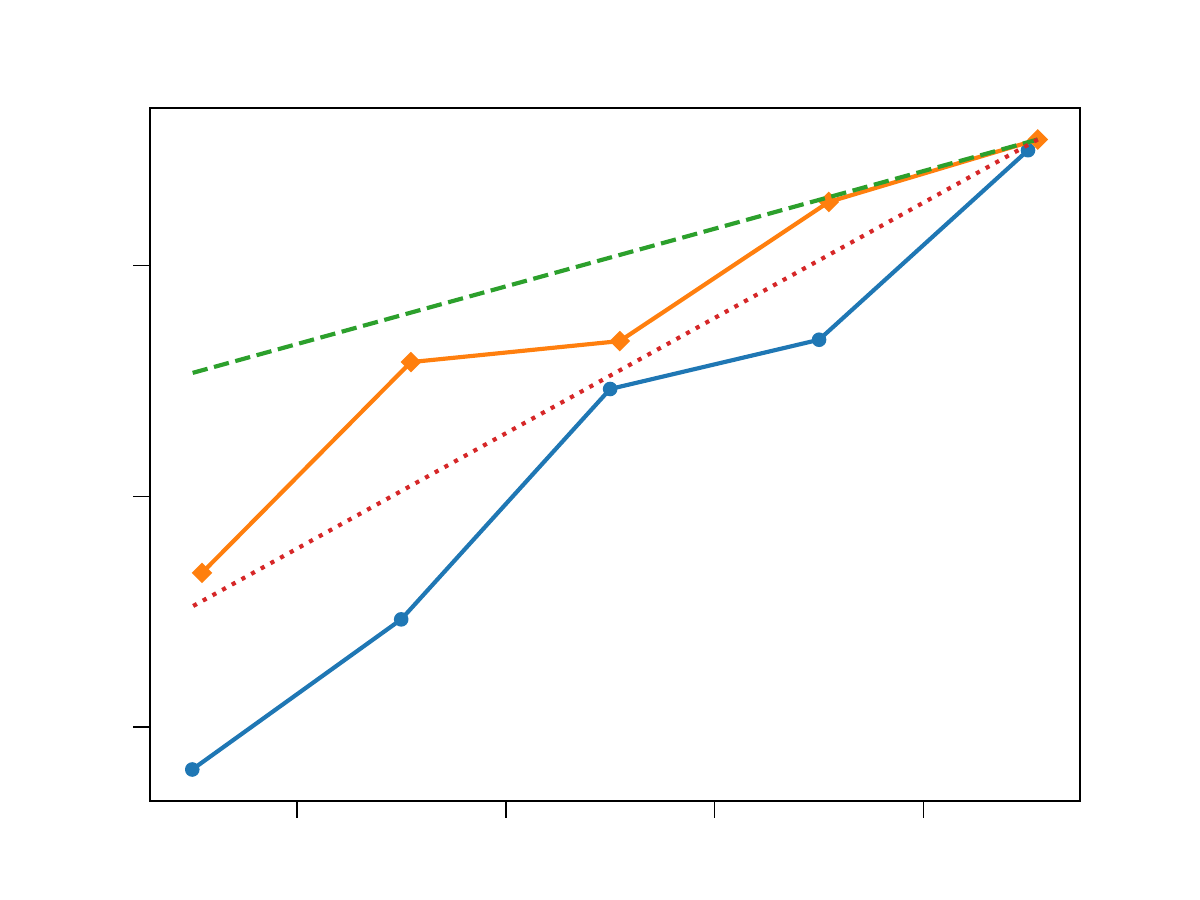}};
		
		\begin{scope}[
			shift={(img.south west)},
			x={($(img.south east)-(img.south west)$)},
			y={($(img.north west)-(img.south west)$)}
			]
			
			\def\xleft{0.14}
			\def\xright{0.93}
			\def\ybottom{0.13}
			\def\ytop{0.88}
			
			\node[anchor=north] at (0.272,\ybottom) {$2^{-5}$};
			\node[anchor=north] at (0.44,\ybottom) {$2^{-4}$};
			\node[anchor=north] at (0.607,\ybottom) {$2^{-3}$};
			\node[anchor=north] at (0.775,\ybottom) {$2^{-2}$};
			
			\node[anchor=east] at (\xleft,0.22) {$2^{-4}$};
			\node[anchor=east] at (\xleft,0.463) {$2^{-3}$};
			\node[anchor=east] at (\xleft,0.706) {$2^{-2}$};
			
			\node[anchor=north] at (0.55,0.04) {$h$};
			\node[anchor=south, rotate=90] at (0.04,0.52) {$\Vert \xi_N - \xi_h\Vert_2$};
			
		\end{scope}
		\node[
		draw,
		fill=white,
		anchor=south east,
		inner sep=1pt
		] at ([xshift=-9mm,yshift=8mm]img.south east)
		{
			\begin{tikzpicture}
				\draw[mplblue, thick] (0,1) -- (0.4,1);
				\filldraw[mplblue] (0.2,1) circle (2pt);
				\node[right] at (0.4,1) {Regular mesh};
				
				\draw[mplorange, thick] (0,0.7) -- (0.4,0.7);
				\filldraw[mplorange]
				(0.2,0.7+0.08) --
				(0.2+0.08,0.7) --
				(0.2,0.7-0.08) --
				(0.2-0.08,0.7) -- cycle;
				\node[right] at (0.4,0.7) {Irregular mesh};
				
				\draw[dashed, thick, mplgreen] (0,0.4) -- (0.4,0.4);
				\node[right] at (0.4,0.4) {$h^{1/4}$ rate};
				
				\draw[dotted, thick,mplred] (0,0.1) -- (0.4,0.1);
				\node[right] at (0.4,0.1) {$h^{1/2}$ rate};
			\end{tikzpicture}
		};
		
	\end{tikzpicture}
	\begin{tikzpicture}[font=\scriptsize]
		\node (img) at (0,0)
		{\includegraphics[width=0.465\textwidth]{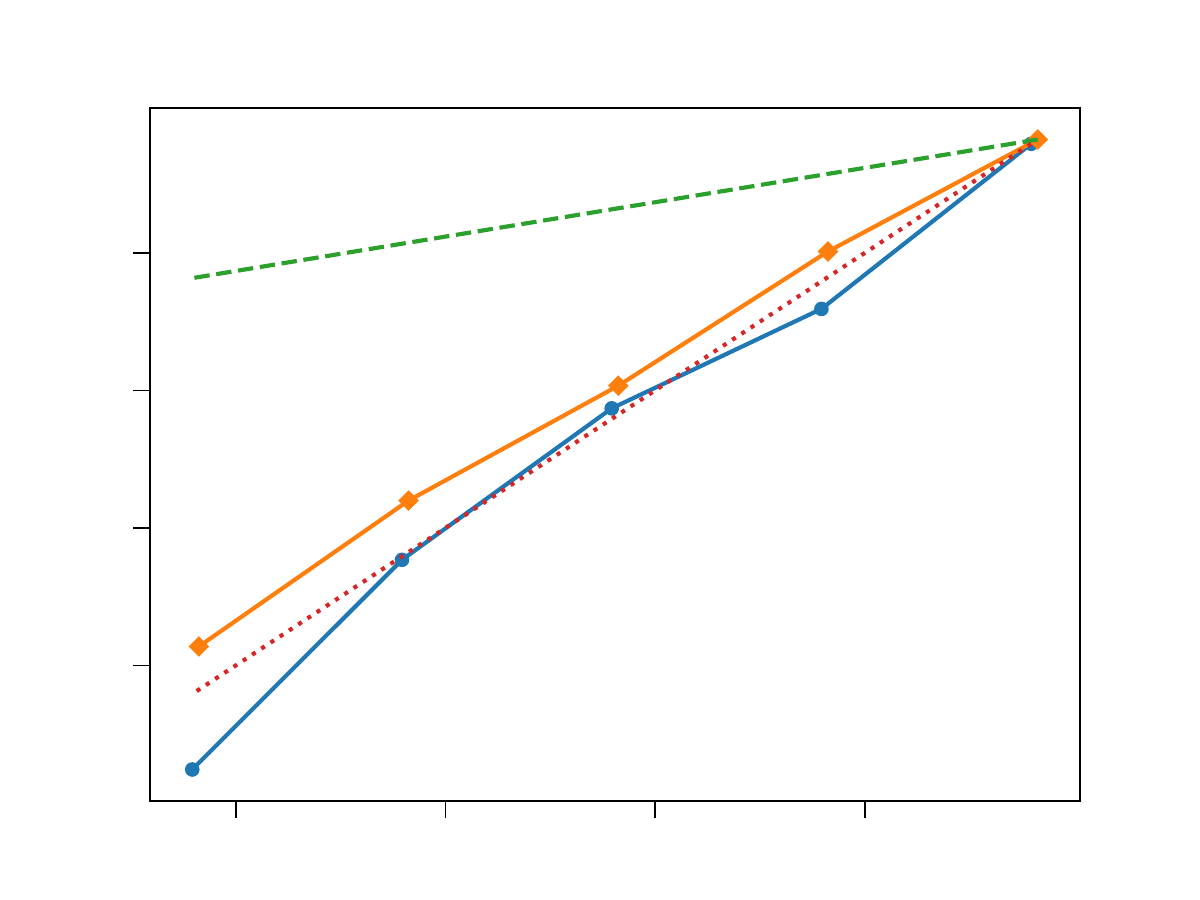}};
		\begin{scope}[
			shift={(img.south west)},
			x={($(img.south east)-(img.south west)$)},
			y={($(img.north west)-(img.south west)$)}
			]
			
			\def\xleft{0.14}
			\def\xright{0.93}
			\def\ybottom{0.13}
			\def\ytop{0.88}
			
			\node[anchor=north] at (0.225,\ybottom) {$2^{-5}$};
			\node[anchor=north] at (0.393,\ybottom) {$2^{-4}$};
			\node[anchor=north] at (0.56,\ybottom) {$2^{-3}$};
			\node[anchor=north] at (0.727,\ybottom) {$2^{-2}$};
			
			\node[anchor=east] at (\xleft,0.29) {$2^{-5}$};
			\node[anchor=east] at (\xleft,0.435) {$2^{-4}$};
			\node[anchor=east] at (\xleft,0.58) {$2^{-3}$};
			\node[anchor=east] at (\xleft,0.725) {$2^{-2}$};
			
			\node[anchor=north] at (0.55,0.04) {$h$};
			
		\end{scope}
		\node[
		draw,
		fill=white,
		anchor=south east,
		inner sep=1pt
		] at ([xshift=-9mm,yshift=8mm]img.south east)
		{
			\begin{tikzpicture}
				\draw[mplblue, thick] (0,1) -- (0.4,1);
				\filldraw[mplblue] (0.2,1) circle (2pt);
				\node[right] at (0.4,1) {Regular mesh};
				
				\draw[mplorange, thick] (0,0.7) -- (0.4,0.7);
				\filldraw[mplorange]
				(0.2,0.7+0.08) --
				(0.2+0.08,0.7) --
				(0.2,0.7-0.08) --
				(0.2-0.08,0.7) -- cycle;
				\node[right] at (0.4,0.7) {Irregular mesh};
				
				\draw[dashed, thick, mplgreen] (0,0.4) -- (0.4,0.4);
				\node[right] at (0.4,0.4) {$h^{1/4}$ rate};
				
				\draw[dotted, thick,mplred] (0,0.1) -- (0.4,0.1);
				\node[right] at (0.4,0.1) {$h$ rate};
			\end{tikzpicture}
		};
		
	\end{tikzpicture}
	\caption{Experimental convergence rates in spatial (left) and spatiotemporal (right) settings, color figure online}
	\label{fig:hconvfig}
\end{figure}

\subsection{Results in the Inverse Problem in Electrocardiographic Imaging}

We now turn to the main application-oriented part of this work: the inverse problem in electrocardiographic imaging (ECGI).
ECGI is an important and actively studied research field concerned with the non-invasive reconstruction of cardiac electrical activity from body-surface measurements, with potential applications in the diagnosis and treatment planning of cardiac arrhythmias.
From a mathematical perspective, ECGI is a challenging PDE-constrained inverse problem, characterized by severe ill-posedness, limited measurements, and pronounced spatiotemporal dynamics.
While the preceding denoising and inpainting experiments provide controlled benchmark settings for validating the theoretical convergence results, ECGI serves as the motivating real-world application for the proposed spatiotemporal regularization framework.
We focus on the reconstruction of the transmembrane potential in the myocardium, first introducing the ECGI forward problem and the corresponding existence and uniqueness results before formulating the inverse problem.
After that, we prove well-posedness of the considered regularization problems and present state-of-the-art Tikhonov and spatiotemporal total variation regularization methods.
As a practical application, we then present three spatiotemporal activation patterns in a $2D$ setting with realistic conductivity assumptions.

The well-posed forward problem of electrocardiography aims to compute the electrical potential on the torso given the transmembrane potential in the myocardium.
Let $\OmegaSpace \subset \R^2$ be the whole body domain with \emph{body surface} $\Gamma = \partial \OmegaSpace$.
In this simplified model, the whole body consists of the \emph{myocardium}, the muscular tissue of the heart, denoted by $\Omega_H \subset \OmegaSpace$  and $\Omega_0 = \OmegaSpace \setminus \overline{\Omega}_H$, the \emph{torso domain}.
Additionally, the heart-torso interface, \emph{epicardium}, is denoted by $\Gamma_H = \partial \Omega_H$.
Both boundaries are assumed to be Lipschitz and bounded, satisfying $\Gamma \cap \Gamma_H = \varnothing$.
All domains are assumed to be static and invariant over time.
\footnote{
	Although the heart is continuously beating, this assumption can be considered a reasonable approximation in the present context.
	The ventricular activation sequence is very fast and occurs before substantial ventricular contraction develops, so cardiac motion during the activation interval is relatively limited.
	Moreover, accurately capturing cardiac motion during ECGI is nontrivial.
	Accounting for motion would render the forward operator time-dependent, resulting in a significant computational cost in terms of storage and assembly.
	Some computational results on the impact of cardiac motion on the ECG are available in~\cite{Mi09, Mo21}.}
We denote by $\sigma \in L^\infty(\Omega_0, \R^+_{d\times d})$ the \emph{torso conductivity tensor} (bulk conductivity) and by $\sigma_i, \sigma_e \in L^\infty(\Omega_H, \R^+_{d\times d})$ the \emph{intra- and extracellular conductivity tensors of the heart}.
All these conductivity tensors are space-dependent and each satisfies the ellipticity condition $\kappa^{-1} \vert \mathbf{y}\vert^2 \leq \sigma(\mathbf{x)}\mathbf{y}\cdot\mathbf{y} \leq \kappa \vert \mathbf{y}\vert^2$ for some $\kappa > 0$ and all $\mathbf{y} \in \R^2$.
Given a finite time interval $\OmegaTime$, \emph{transmembrane potential} $u\colon \Omega_H \times \OmegaTime\rightarrow \R$
and \emph{body potential} $v\colon \OmegaSpace \times \OmegaTime\rightarrow \R$ we can formulate the forward problem as

\begin{equation}
	\label{eq:transmembrane_fwd}
	\begin{cases}
		\begin{aligned}
			- \Div ((\sigma_i + \sigma_e)(\mathbf{x}) \nabla_x v(\mathbf{x},t)) &= \Div(\sigma_i(\mathbf{x}) \nabla_x u(\mathbf{x},t)) , & (\mathbf{x}, t)&\in \Omega_H \times \OmegaTime\\
			- \Div (\sigma(\mathbf{x}) \nabla_x v(\mathbf{x},t)) &=0 , & (\mathbf{x}, t)&\in \Omega_0 \times \OmegaTime\\
			\sigma(\mathbf{x}) \nabla_x v(\mathbf{x},t) \cdot \nu &= 0 , & (\mathbf{x}, t)&\in \Gamma \times \OmegaTime\\
			v(\mathbf{x}^-,t) &= v(\mathbf{x}^+, t) , & (\mathbf{x}, t)&\in \Gamma_H \times \OmegaTime\\
			\sigma(\mathbf{x}^+)\nabla v(\mathbf{x}^+, t) \cdot\nu - \\
			(\sigma_i + \sigma_e)(\mathbf{x}^-)\nabla v(\mathbf{x}^-, t) \cdot \nu &= \sigma_i(\mathbf{x}^-)\nabla u(\mathbf{x}^-, t)\cdot \nu  & (\mathbf{x}, t)&\in \Gamma_H \times \OmegaTime.
		\end{aligned}
	\end{cases}
\end{equation}
In the equation, $\nu$ denotes the respective outer unit normal vector, and $\mathbf{x}^-, \mathbf{x}^+$ denote the inner and outer limits towards the boundary $\Gamma_H$.
See \cref{fig:domain} for a visual explanation of the domain.
It is possible to show the existence and regularity of the so-called forward solutions of this elliptic PDE system.
\begin{theorem}
	In a weak sense, there exists a solution $v \in L^2\left(\OmegaTime, H^1(\OmegaSpace)\right)$ of \eqref{eq:transmembrane_fwd} given $u \in L^2\left(\OmegaTime, H^1(\Omega_H)\right)$.
	The solution is unique up to the addition of constants in space.
	\label{thm:fwdsol_transmembrane}
\end{theorem}
\begin{proof}
	The proof follows from standard arguments in elliptic PDE theory.
	For more details, we refer to~\cite{Ev98}.
\end{proof}
Hence, we can define a \emph{forward operator} associating $u$ with a solution $v_u\colon \OmegaSpace\rightarrow\R$.
Due to the invariance w.r.t. spatial constants, any solution expressed in terms of $u$ and $v$ in the following is taken to be the spatial zero-mean representative, i.e., $v\colon\frac{1}{\vert\OmegaSpace\vert} \int_{\OmegaSpace} v \dx x = 0$.
With the help of the trace theorem, we can define the continuous forward operator 
\begin{equation*}
	A\colon L^2\left(\OmegaTime, H^1(\Omega_H)\right)\rightarrow L^2\left(\Sigma\times\OmegaTime \right)\qquad u\mapsto v_u \vert_\Sigma,
\end{equation*}
transferring the transmembrane potential in the heart to the electrical potential on the torso surface.
$A$ is not injective, and its generalized inverse is unbounded.

For the inverse problem, the \emph{body surface potential} is known on the electrodes, represented as a subset $\Sigma \subset \Gamma$ of $N_\Sigma \in \N$ electrodes.
We assume that these do not collapse to a single point and are pairwise disjoint.
Hence, we can represent the body surface potential as a function $z \in L^2(\Sigma \times \OmegaTime)$ and define the data fidelity term
\begin{equation*}
	{G}(u)\coloneqq \frac{1}{2 N_\Sigma} \int_{\Sigma \times \OmegaTime} \left( A\lbrack u\rbrack(\mathbf{x}, t) - z(\mathbf{x}, t)\right)^2 \dx (\mathbf{x}, t),
\end{equation*}
where we divide by the number of electrodes $N_\Sigma$ to keep weighting between data fidelity term and regularizer consistent across different $N_\Sigma$.
Due to the typical sparsity of the electrodes and the ellipticity of the forward operator, this problem is severely ill-posed and necessitates regularization.
Additionally, the problem is only defined for weakly differentiable functions; hence, we cannot use piecewise constant discretizations later.
As in~\cite{Ha25}, small modifications of the regularized functional allow for an existence proof in the continuous setting.
In the existence proof, we consider the half-quadratic setting, so we define
\begin{align*}
	\widetilde{\TGV}^{2, 1}_{\varepsilon, (\alpha, \beta)}(u) &\coloneqq \min_w\int_{\OmegaTime} \alpha_1\zeta_\varepsilon(\nabla_\mathbf{x}u - w)  + \alpha_0\zeta_\varepsilon(\mathcal{E}_\mathbf{x}w)\dx t+ \int_{\Omega_H} \beta \zeta_\varepsilon(\nabla_tu) \dx \mathbf{x},\\
	\zeta_\varepsilon(y) &\coloneqq \begin{cases}  \vert y\vert_2, \quad& \vert y\vert_2 \leq \frac{1}{\varepsilon}\\
		\varepsilon \vert y\vert_2^2, \quad&\text{else.}
	\end{cases} 
\end{align*}
\begin{theorem}
	Given $\lambda >0$ and a Tikhonov regularizer $T(u) = \frac{\lambda}{2}\Vert u \Vert_{L^2(\Omega_H\times \OmegaTime)}^2$ ${J}(u) = {G}(u) +  \widetilde{\TGV}^{2, 1}_{\varepsilon, (\alpha, \beta)}(u) + T(u)$ admits a unique minimizer in $L^2\left(\OmegaTime, H^1(\Omega_H)\right)$ given $z\in L^2(\Sigma\times\OmegaTime)$.
\end{theorem}
\begin{proof}
	We use the direct method in the calculus of variations.
	Along a minimizing sequence, $J$ is bounded. The Tikhonov term bounds $u$ in $L^2$. At every
	time $t$, the spatial term bounds $\nabla_\mathbf{x}u(t)-w(t)$ and $\mathcal{E}_\mathbf{x}w(t)$ in $L^2$
	for the optimal $w(t)$, since $\zeta_\varepsilon(y)\geq\varepsilon\vert y\vert_2^2$.
	Korn's inequality controls $w(t)$ up to rigid motions, and the Tikhonov term controls the remaining
	finite-dimensional kernel, on which the spatial term vanishes (e.g.\ for $u$ affine in space).
	Hence minimizing sequences are bounded in $L^2(\OmegaTime;H^1(\Omega_H))$, and a weakly convergent
	subsequence exists.
	All terms are convex.
	The spatial term is the time integral of a convex, continuous function of
	$u(t)$, and the temporal term is a convex, nondecreasing function of the temporal total
	variation, which is lower semicontinuous on $L^1(\OmegaTime)$. Hence both are weakly lower
	semicontinuous, and a minimizer exists.
	Uniqueness follows from the strict convexity of the Tikhonov term.
	Finiteness of the temporal term implies $u(\mathbf{x},\cdot)\in \BV(\OmegaTime)$ for a.e.\
	$\mathbf{x}\in\Omega_H$.
\end{proof}
Note that the theorem also holds for $\widetilde{\NSymTGV}_{\varepsilon, (\alpha, \beta)}^{(2, 1)}$ by using Poincar\'e's instead of Korn's inequality.

To use the primal-dual algorithm, we again compute the proximal map of $G$,
\begin{equation*}
	\prox_{\tau G}(\widetilde{u})=\mathrm{argmin}_{u_h \in L^2(\Omega_H\times\OmegaTime)} \frac{\tau}{2 N_\Sigma}\Vert Au - z\Vert_{L^2(\Sigma \times \OmegaTime)} ^2+  \frac{1}{2} \Vert u - \widetilde{u}\Vert_{L^2(\Omega_H\times\OmegaTime)}^2 .
\end{equation*}
After differentiation and solving for the first-order optimality condition, we obtain
\begin{equation*}
	\prox_{\tau G}(\widetilde{u}) = \left(\frac{\tau}{N_\Sigma}A^TA + I\right)^{-1}\left(\frac{\tau}{N_\Sigma}A^Tz + \widetilde{u}\right).
\end{equation*}
We obtain the same proximal map up to the addition of mass matrices in the discrete setting.

The most practically used regularization for ECGI remains Tikhonov regularization~\cite{Li25, Ti77}, defined as
$
\frac{\alpha}{2}\Vert Lu\Vert_{L^2}^2.
$
The linear operator $L$ is either the identity matrix $I$ for zero-order Tikhonov, denoted as $\mathbf{T0}$, or the gradient $\nabla_{x, t}$ for first-order Tikhonov in space and time, $\mathbf{T1}$.
It is possible to explicitly compute a solution to $\min_u \mathbf{G}(u) + \mathbf{T0}$ for each discrete timestep $t_s\in\N$ given a spatial regularization parameter $\alpha_0$ as
\begin{equation*}
	u(t_s) = \left( \frac{1}{N_\Sigma} A^*A + \alpha_0 I\right)^{-1}\frac{1}{N_\Sigma}A^*z(t_s).
\end{equation*}
For the latter, we can use the time-stepping solution as described in~\cite[Section 5.1]{Ha25} after fixing initial condition $u(-1)=0$ and disjoint spatial and temporal regularization parameters $(\alpha_0, \beta)$.
More explicitly, this is
\begin{equation*}
	u(t_s) = (\frac{1}{ N_\Sigma}A^*A - \alpha_0 \Delta_{\x, h} + \beta I)^{-1}(\frac{1}{N_\Sigma}A^*z(t_s)+\beta u(t_{s-1})).
\end{equation*}
with spatial Laplacian $\Delta_{\x,h}$.

We also compare to total variation in space and time; given spatial and temporal weights $(\alpha, \beta) > 0$ as well as $\gamma = 1, 2$ it can be defined as
\begin{equation*}
	\mathbf{TV_{P1}^\gamma}(u) \coloneqq\int_{\OmegaSpace \times \OmegaTime} \alpha \vert \nabla_\mathbf{x}u(\mathbf{x}, t)\vert_\gamma + \beta\vert \nabla_\mathbf{t}u(\mathbf{x}, t)\vert_\gamma \dx (\mathbf{x}, t).
\end{equation*}
For $\gamma=1$, this approach was recently applied in the ECGI setting~\cite{Ha25}, whereas for $\gamma=2$, the original authors employed a joint spatiotemporal approach.
We instead use separate integrals, which provide a more direct comparison with $\mathbf{TGV_{P1}^2}$, where the integrals are likewise treated separately.
Defined as such, the regularizer uses the same discrete gradient space as $\mathbf{TGV_{P1}}$, $\mathcal{Q}_{\mathrm{ST}, h}$ \eqref{eq:defSTSpace}.
Hence, we can reuse the primal-dual optimization algorithm~\cref{alg:algorithm} after removing the second-order primal variable $w^n$ and related dual variable $q^n$.
For more details regarding the optimization and exact discrete formulation, we refer to the original paper.

Note that the transmembrane potential enters~\eqref{eq:transmembrane_fwd} only through its spatial gradient, both as a volume source in $\Omega_H$ and as a flux on $\Gamma_H$.
Testing with $\varphi \in H^1(\OmegaSpace)$ and using the interface conditions, the weak formulation at each time $t$ reads
\begin{equation*}
	\int_{\Omega_H} (\sigma_i + \sigma_e) \nabla_\mathbf{x} v \cdot \nabla \varphi \dx \mathbf{x} + \int_{\Omega_0} \sigma \nabla_\mathbf{x} v \cdot \nabla \varphi \dx \mathbf{x} = - \int_{\Omega_H} \sigma_i \nabla_\mathbf{x} u \cdot \nabla \varphi \dx \mathbf{x},
\end{equation*}
whose right-hand side requires $u(\cdot, t) \in H^1(\Omega_H)$~\cite{Bu10, Wa13}.
For a piecewise constant function, $\nabla_\mathbf{x} u$ is a measure concentrated on the interior edges of the mesh rather than an $L^2$ function, so the forward operator $A$ is not defined on such functions in the variational sense of~\cref{thm:fwdsol_transmembrane}.
Consequently, $\mathbf{TGV_{P0}}$ is not applicable in this setting, whereas the $H^1$-conforming discretization $\mathcal{V}_{\mathrm{ST},h}$ is admissible by construction.
\begin{figure}[htp]
	\centering{
		\begin{tikzpicture}
			\node (torso) at (0,0) {\includegraphics[width=8cm]{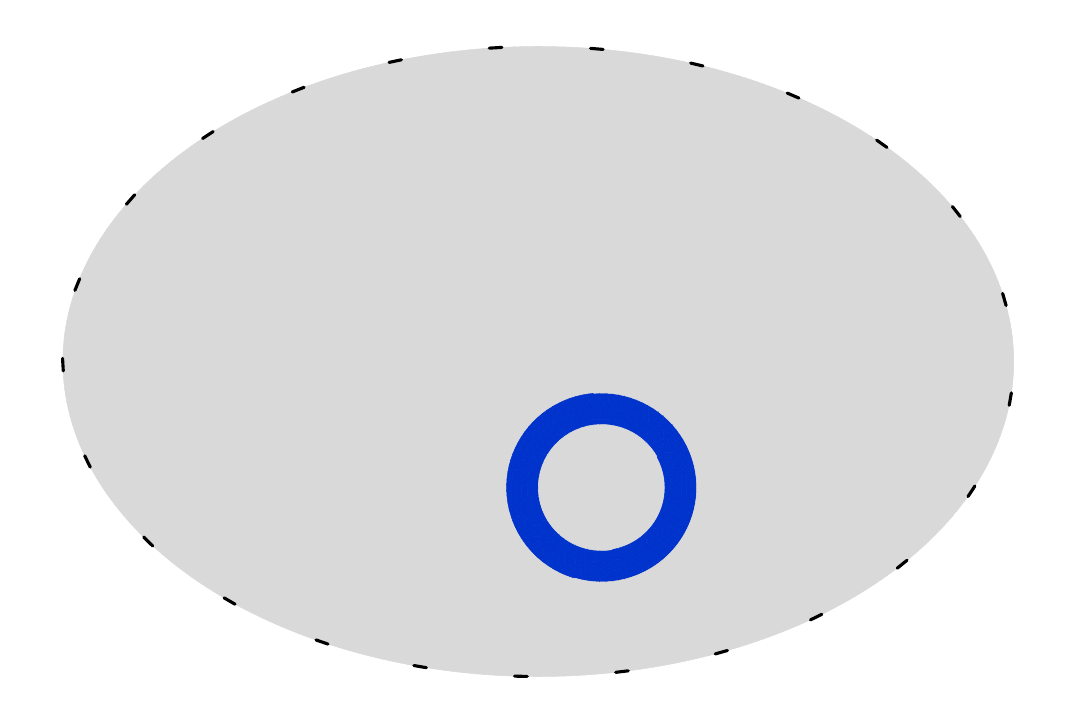}};
			\coordinate (heartpos) at ($(torso.center)+(0.5,-0.97)$);
			\node (heart) at (6.3,0) {\includegraphics[width=5cm]{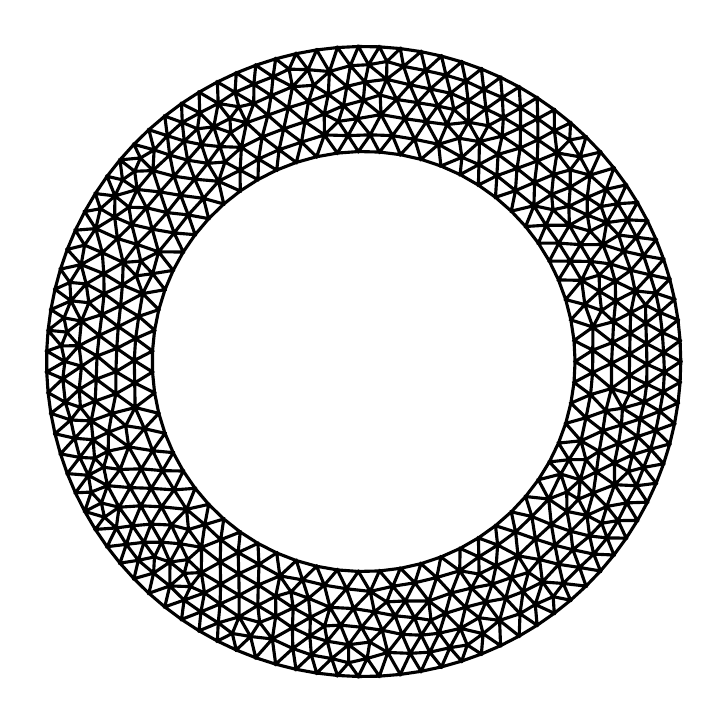}};
			\draw[thick] 
			($(heartpos)+(0.24,0.7)$) -- ($(heart.west)+(1.7,2.05)$);
			
			\draw[thick] 
			($(heartpos)+(0.24,-0.65)$) -- ($(heart.west)+(2.2,-2.2)$);
			\coordinate (sigma) at ($(torso.south west)+(1,1)$);
			
			\draw ($(sigma)+(-.25,1.02)$) -- ($(sigma)+(.0,.2)$);
			\draw ($(sigma)+(.2,.4)$) -- ($(sigma)+(.1,.2)$);
			\draw ($(sigma)+(0.8,-.0)$) -- ($(sigma)+(.15,0.05)$);
			\draw ($(sigma)+(+1.47,-.3)$) -- ($(sigma)+(.15,-.05)$);
			
			\node at (sigma) {$\Sigma$};
			
			\coordinate (gamma) at ($(torso.west)+(1,2)$);
			\draw ($(torso.west)+(2.,1.9)$) -- ($(gamma)+(.1,0)$);
			\node at (gamma) {$\Gamma$};
			
			\coordinate (omega) at ($(torso.center)+(-.5, .5)$);
			\node at (omega) {$\Omega_0$};
			
			\coordinate (omega_h) at ($(torso.center)+(1.6, -0.4)$);
			\node at (omega_h) {$\Omega_H$};
			\draw ($(omega_h)+(-.25,0)$) -- ($(omega_h)+(-.7,-.1)$);
			
			\coordinate (gamma_h) at ($(torso.center)+(-.5,-1.4)$);
			\draw ($(gamma_h)+(.05,.1)$) -- ($(gamma_h)+(.28,.18)$);
			\draw ($(gamma_h)+(.15,.05)$) -- ($(gamma_h)+(.6,.1)$);
			\node at (gamma_h) {$\Gamma_H$};
			
		\end{tikzpicture}
	}
	\caption{Left: Used torso-heart model with highlighted heart domain $\Omega_H$, torso domain $\Omega_0$, epicardium $\Gamma_H$, torso boundary $\Gamma$ and 25 electrodes $\Sigma$.
		Right: Heart mesh used in the numerical experiments.}
	\label{fig:domain}
\end{figure}
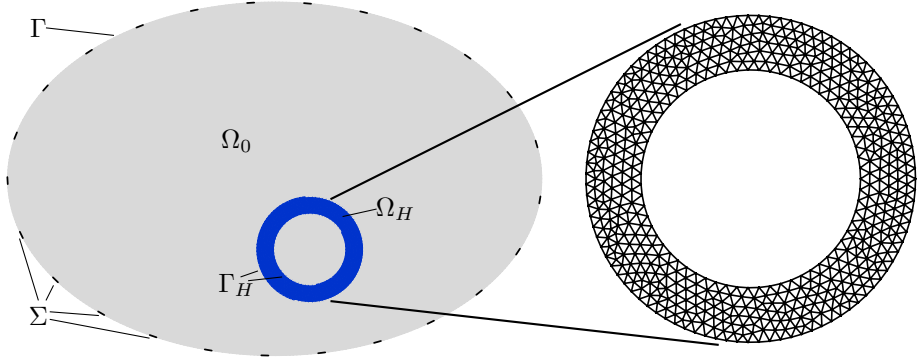

To test the reconstructions in the ECGI setting, we employ three model ground truth functions with different features to compare the reconstruction properties of the regularizers.
All reconstructions are performed on a 2D cardiac disk embedded in an ellipsoid representing the torso.
The cardiac transfer operator was computed using Firedrake~\cite{FiredrakeUserManual} and petsc4py~\cite{Da11}.
Since constants generate no source in~\eqref{eq:transmembrane_fwd} and are invariant under the spatial regularizers, the functional is invariant under adding a global constant to $u$; time-dependent constants are penalized only by the temporal regularizer.
Rather than adding the Tikhonov term $T$, which in our experiments either had no visible effect or biased the reconstructions strongly towards zero, we remove this constant explicitly by shifting the noisy data to zero mean over all electrodes and time steps and the reconstruction to zero mean over $\Omega_H \times \OmegaTime$.
We reconstruct a function that is linear in space and piecewise constant in time~\cref{eq:LinearToyFunction}, a quadratic function that rotates along the disk in time~\cref{eq:QuadraticToyFunction}, and a dampened $\tanh$ function where the zero level set again rotates~\cref{eq:TanhToyFunction}.
Every function is simulated and reconstructed on a mesh with 572 points over 1001 timesteps.
More explicitly, they are defined as
\begin{align}
	u(r, \phi, t) &= 2(\bar{\phi} + \lfloor 10t\rfloor/10 \mod 1)(\bar{r}+3)/4\label{eq:LinearToyFunction} - 1,\\
	u(x, y, t) &= (x\cos(2\pi t) + y \sin(2\pi t)\label{eq:QuadraticToyFunction})^2 - (y\cos(2\pi t) - x \sin(2\pi t))^2,\\
	u(r, \phi, t) &= \tanh(10(1-t-\vert 1 - \bar{\phi}\vert))\label{eq:TanhToyFunction},
\end{align}
where $\bar{\phi}, \bar{r}$ denote the angular and radial coordinates normalized to $\lbrack 0, 1\rbrack$.
Plots of the sample timestep $t=250$ can be found in the left column of \cref{fig:cardiac_recon_comparison}.
In accordance with the literature~\cite{El17}, we add white Gaussian noise $\mathbf{n} \sim \mathcal{N}(0, \sigma^2I)$ for a \emph{signal-to-noise ratio (SNR)} level of 50 to the simulated electrode data.
Given $\mathbf{z}^{gt}=Au^{gt}$ and its average value $\overline{\mathbf{z}_t^{gt}}$ over each timestep $t$, $\sigma^2$ is chosen as $\overline{(\mathbf{z}_t^{gt})^2}/ 10^5$.
Results are compared in relative error (RE) and correlation coefficient (CC) defined as 
$
\text{RE}(\mathbf{u}) = \frac{\Vert{\mathbf{u}-\mathbf{u}^\mathrm{g}}\Vert_{2}}{\Vert{\mathbf{u}^\mathrm{g}}\Vert_{2}}, \quad
\text{CC}(\mathbf{u}) = \frac{(\mathbf{u}-\overline{\mathbf{u}},\mathbf{u}^\mathrm{g}-\overline{\mathbf{u}}^\mathrm{g})_{2}}{\Vert{\mathbf{u}-\overline{\mathbf{u}}\Vert}_{2}\Vert{\mathbf{u}^\mathrm{g}-\overline{\mathbf{u}}^\mathrm{g}\Vert}_{2}},
$
where $\mathbf{u}^\mathrm{g}$ represents the nodal values of the ground truth function and $\mathbf{u}$ the nodal values of the reconstruction.
\begin{figure}
	\centering
	\begin{tikzpicture}
		
		\node(grid)at (0,0){
			\begin{tabular}{c c c c c}
				& \textbf{GT data} & $\mathbf{T1}$ & $\mathbf{TV_{P1}^2}$ & $\mathbf{TGV_{P1}^2}$ \\
				
				\raisebox{17pt}{\rotatebox{90}{Linear}} &
				\includegraphics[width=0.17\textwidth]{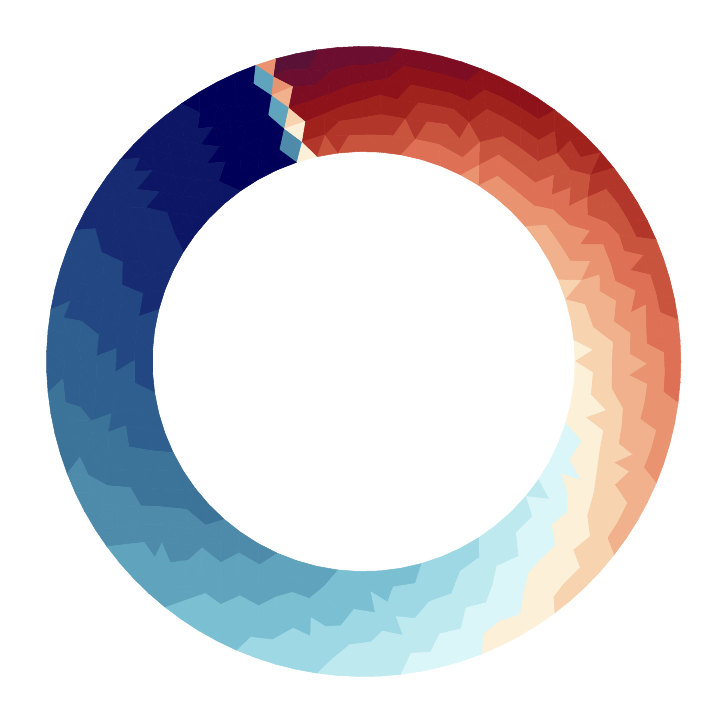} &
				\includegraphics[width=0.17\textwidth]{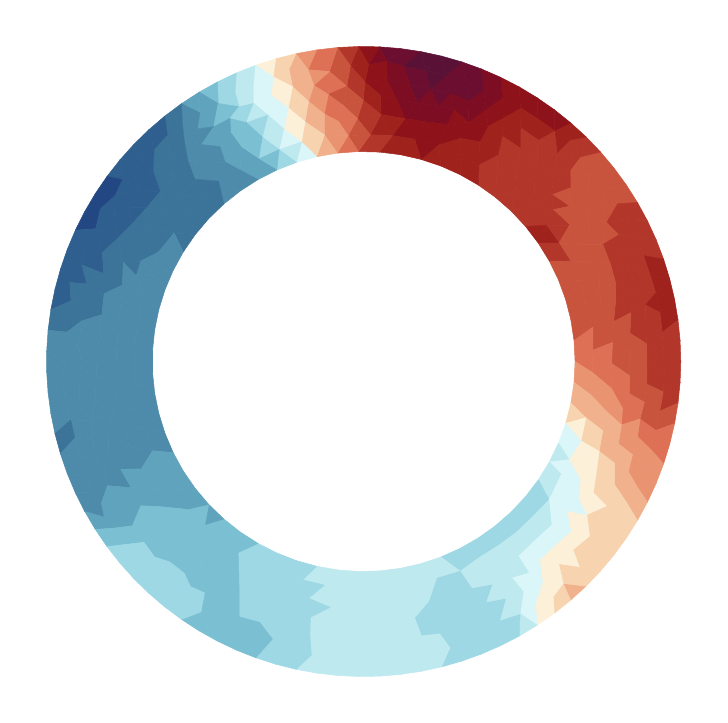} &
				\includegraphics[width=0.17\textwidth]{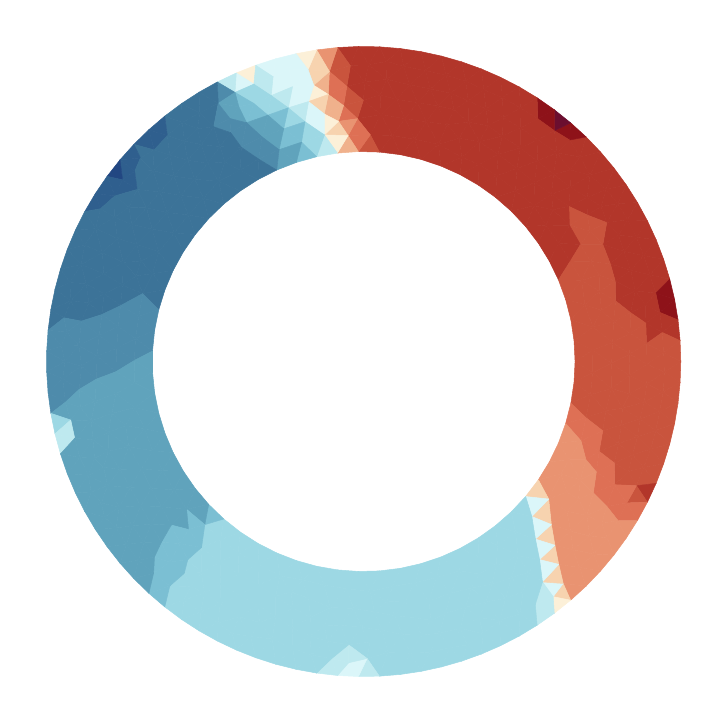} &
				\includegraphics[width=0.17\textwidth]{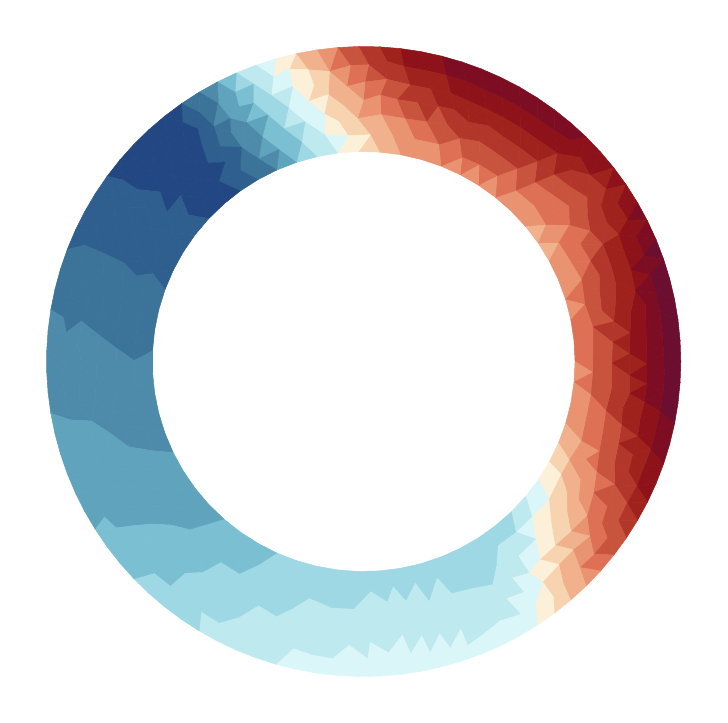} \\
				
				\raisebox{10pt}{\rotatebox{90}{Quadratic}} &
				\includegraphics[width=0.17\textwidth]{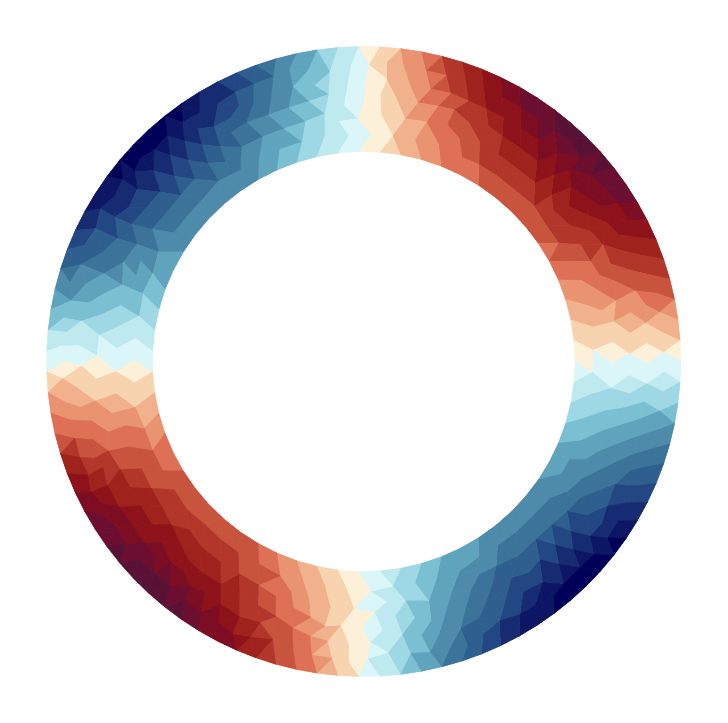} &
				\includegraphics[width=0.17\textwidth]{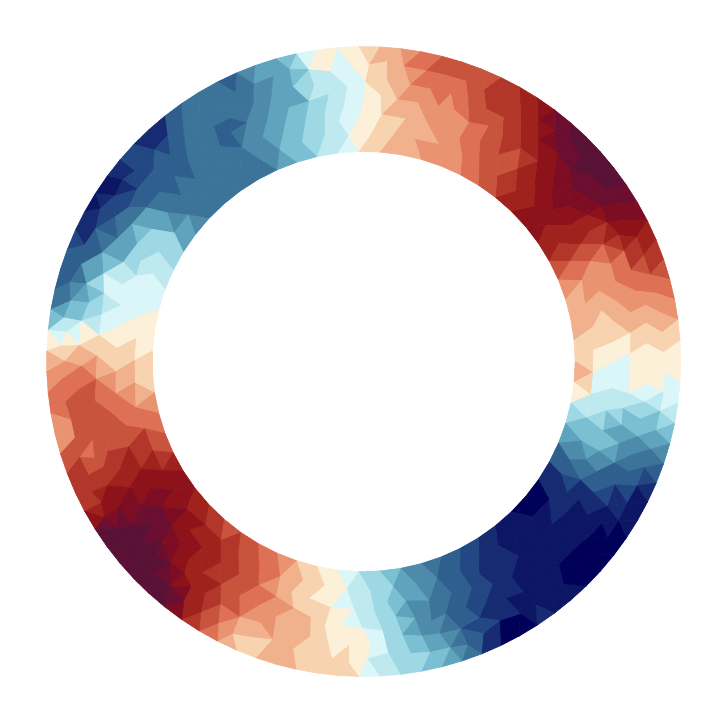} &
				\includegraphics[width=0.17\textwidth]{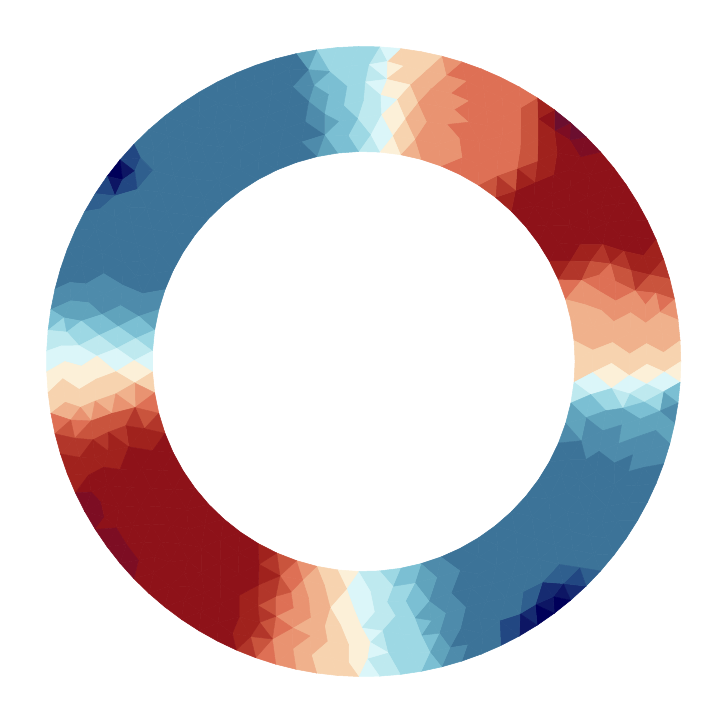} &
				\includegraphics[width=0.17\textwidth]{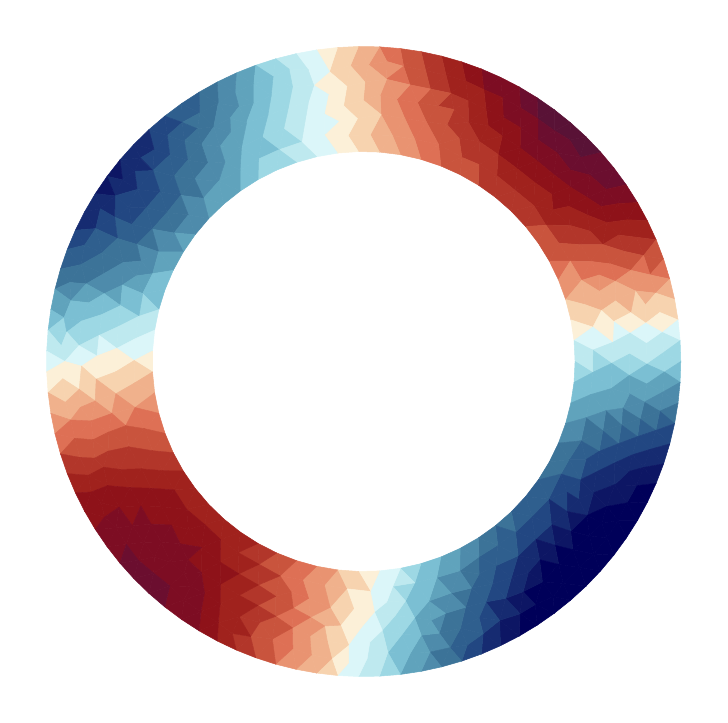} \\
				
				\raisebox{20pt}{\rotatebox{90}{Tanh}} &
				\includegraphics[width=0.17\textwidth]{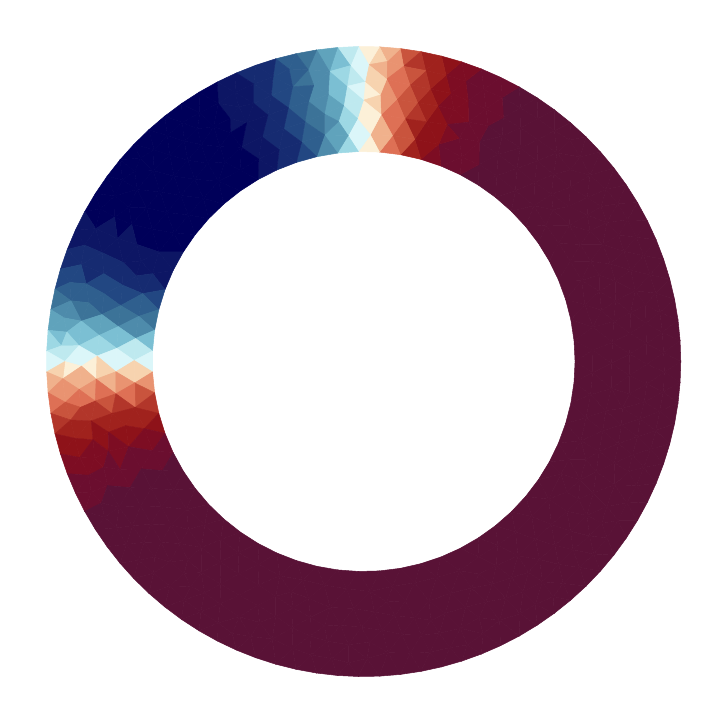} &
				\includegraphics[width=0.17\textwidth]{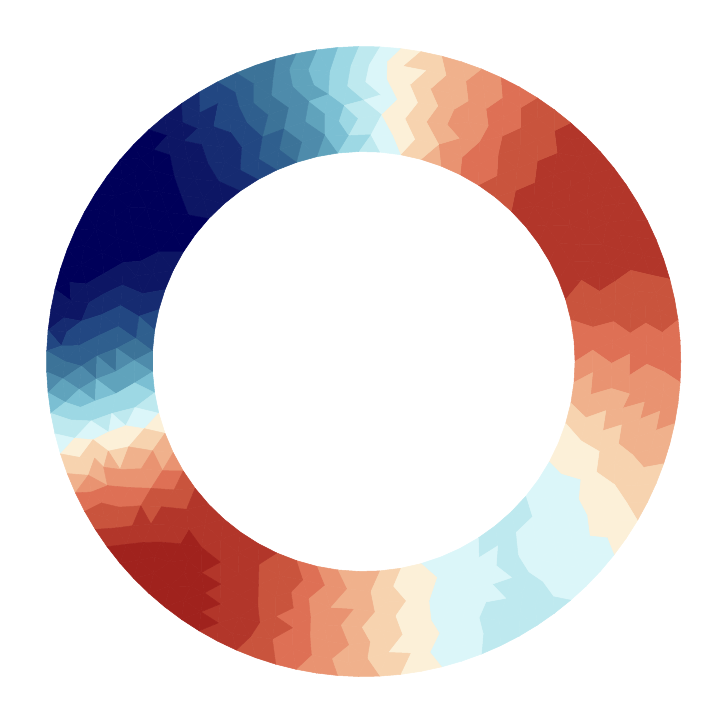} &
				\includegraphics[width=0.17\textwidth]{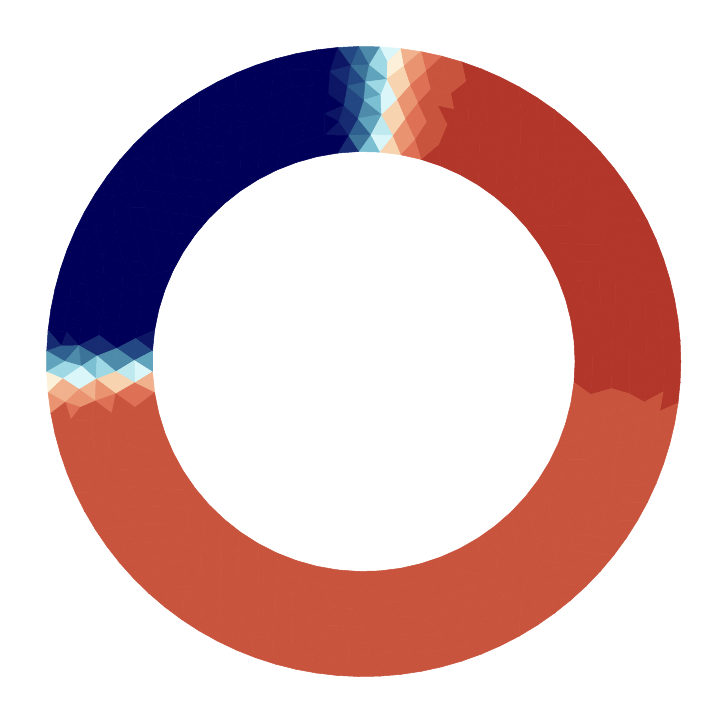} &
				\includegraphics[width=0.17\textwidth]{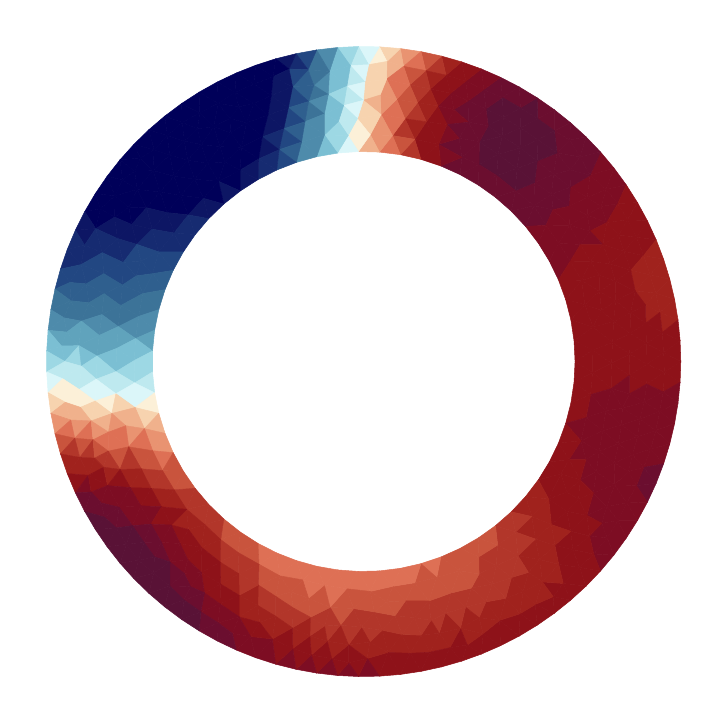} \\
			\end{tabular}
		};
		
		\node[inner sep=0pt](cb) at (6.8,-0.2){\includegraphics[height=0.48\textwidth]{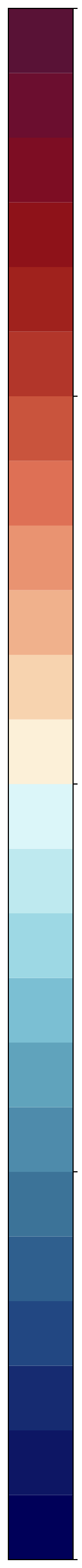}};
		
		\node[right] at ($(cb.south west)!0!(cb.north west)+ (0.3, 0)$) {$-1$};
		\node[right] at ($(cb.south west)!0.25!(cb.north west)+ (0.3, 0)$) {$-0.5$};
		\node[right] at ($(cb.south west)!0.5!(cb.north west)+ (0.3, 0)$) {$0$};
		\node[right] at ($(cb.south west)!0.75!(cb.north west) + (0.3, 0)$) {$0.5$};
		\node[right] at ($(cb.south west)!1!(cb.north west)+ (0.3, 0)$) {$1$};
		
	\end{tikzpicture}
	\caption{
		Comparison of different reconstruction methods at timestep 250.
		Rows correspond to different test cases, while columns show (from left to right) the ground truth data, $\mathbf{T1}$, $\mathbf{TV_{P1}^2}$, and $\mathbf{TGV_{P1}^2}$ reconstructions, color figure online.
	}
	\label{fig:cardiac_recon_comparison}
\end{figure}

Regarding computational time, the explicit reconstructions regularized with $\mathbf{T0}$ and $\mathbf{T1}$ require, on average across the three experiments, $0.16$ and $1.26$ seconds, respectively.
The runtime of the iterative $\mathbf{TV_{P1}}$- and $\mathbf{TGV_{P1}}$-regularized reconstructions depends primarily on the number of iterations required for convergence.
This number is influenced by the parameters $(\alpha,\beta)$, which affect the proximal mappings and, consequently, the $\ell^\infty$-distance between successive iterates.
Averaged over the three examples, the runtimes are $19.29$ seconds for $\mathbf{TV_{P1}^1}$, $22.55$ seconds for $\mathbf{TV_{P1}^2}$, $39.63$ seconds for $\mathbf{TGV_{P1}^1}$, and $31.12$ seconds for $\mathbf{TGV_{P1}^2}$.
Per iteration, the $\mathbf{TGV_{P1}}$ methods require approximately $50\%$ more time than their $\mathbf{TV_{P1}}$ counterparts.
This is expected, as the spatiotemporal $\mathbf{TGV_{P1}}$ formulation requires the evaluation of three proximal mappings, compared with two for $\mathbf{TV_{P1}}$.

The visual comparison in \cref{fig:cardiac_recon_comparison} shows that both $\mathbf{TV_{P1}^2}$ and $\mathbf{TGV_{P1}^2}$ are successful in removing noise from the reconstruction. 
Additionally, $\mathbf{TGV_{P1}^2}$ is able to recover linear and quadratic regions where $\mathbf{TV_{P1}^2}$ suffers from staircasing artifacts.
Methods $\mathbf{TV_{P1}^1}$ and $\mathbf{TGV_{P1}^1}$ exhibit similar patterns and are hence not shown in the figure.
While all methods seem to prefer similar regularization parameters for the linear \eqref{eq:LinearToyFunction} and quadratic \eqref{eq:QuadraticToyFunction} test cases, this is not the case in the $\tanh$ example \eqref{eq:TanhToyFunction}.
Quantitatively, $\mathbf{T0}$ performs much worse than the other methods.
This is because the transfer operator $A$ heavily emphasizes points on the boundary of $\Omega_H$.
Since $\mathbf{T0}$ does not penalize differences between neighboring points, it fails to reconstruct volumetric information.
For the 50dB noise example tested, $\mathbf{TGV_{P1}}$ outperforms all comparison methods on each test case.
It achieves a relative gain over $\mathbf{T1}$ between $12.5$ and $43.1\%$ in terms of $\mathrm{RE}$ and $4.7$ to $31.9\%$ in terms of $\mathrm{CC}$ across tested examples.
The greatest improvement is achieved in the quadratic test case, a second-order function.
Full results can be seen in \cref{tab:2D_table_inverse}.

\begin{table}[htbp]
	\centering
	\resizebox{\textwidth}{!}{\begin{tabular}{|c|c|c|c|c|c|c|}
			\hline
			\diagbox[width=5em]{Case}{Reg.} & \textbf{T0} & \textbf{T1} & $\mathbf{TV^1_{P1}}$& $\mathbf{TV^2_{P1}}$ & $\mathbf{TGV^1_{P1}}$& $\mathbf{TGV^2_{P1}}$ \\
			\hline
			Linear & $0.93/0.37$ & $0.56/0.85$ & $0.55/0.86$& $0.55/0.86$& $\mathbf{0.48/0.90}$ & $0.49/0.89$ \\
			\cline{2-7}
			$\alpha_0$
			& $5.48\cdot10^{-8}$ 
			& $9.75\cdot10^{-10}$
			& $3.81\cdot10^{-10}$
			& $4.32\cdot10^{-10}$
			& $2.83\cdot10^{-10}$ 
			& $3.49\cdot10^{-10}$\rule{0pt}{1em}\\
			$\alpha_1$
			& 
			& 
			& 
			& 
			& $3.80\cdot10^{-10}$
			& $2.85\cdot10^{-9}$ \\
			$\beta$
			&
			& $3.70\cdot 10^{-9}$
			& $2.03\cdot10^{-9}$
			& $1.75 \cdot10^{-10}$
			& $8.53\cdot10^{-10}$
			& $1.60 \cdot 10^{-9}$\\
			\hline 
			Quadratic
			&$0.88/0.48$ & $0.51/0.86$ &$0.59/0.81$& $0.60/0.80$& $0.30/0.95$ & $\mathbf{0.29}/\mathbf{0.96}$ \\
			\cline{2-7}
			$\alpha_0$
			& $3.65\cdot10^{-8}$
			& $4.27\cdot10^{-11}$
			& $1.18\cdot10^{-10}$
			& $2.46\cdot10^{-10}$
			& $1.00\cdot10^{-11}$
			& $1.83\cdot10^{-10}$ \rule{0pt}{1em}\\
			$\alpha_1$
			&
			& 
			&  
			& 
			& $4.88\cdot10^{-9}$ 
			& $1.32\cdot10^{-8}$ \\
			$\beta$
			& 
			&$3.45\cdot10^{-9}$
			&$6.45\cdot10^{-10}$
			&$4.18\cdot10^{-10}$
			&$1.40\cdot10^{-10}$
			&$4.70\cdot10^{-10}$\\
			\hline 
			Tanh
			&$0.97/0.22$ & $0.76/0.69$ & $0.60/0.84$& $0.61/0.83$& $0.47/\mathbf{0.91}$ & $\mathbf{0.46/0.91}$ \\
			\cline{2-7}
			$\alpha_0$
			& $1.41\cdot10^{-6}$
			& $1.12\cdot10^{-7}$
			& $1.20\cdot10^{-7}$
			& $2.66\cdot10^{-7}$ 
			& $1.04\cdot10^{-11}$ 
			& $8.52\cdot10^{-12}$ \rule{0pt}{1em}\\
			$\alpha_1$
			&
			& 
			& 
			& 
			&$3.30\cdot10^{-5}$
			&$4.23\cdot10^{-6}$  \\
			$\beta$
			&
			&$1.80\cdot10^{-11}$
			&$1.64\cdot10^{-7}$
			&$1.72\cdot10^{-7}$
			&$2.93\cdot10^{-8}$ 
			&$3.10\cdot10^{-8}$ \\
			\hline 
	\end{tabular}}
	\caption{Relative error/correlation coefficient of different reconstructions}
	\label{tab:2D_table_inverse}
\end{table}

\section{Conclusions and Outlook}
In this work, we extended the best-known convergence rates for $\TV$ to $\NSymTGV^2_\alpha$ discretized via a piecewise linear finite element space.
After introducing the mathematical background, we proved the convergence of minimizers in the denoising problem with rate $h^{1/4}$.
We then extended the method to the spatiotemporal setting and transferred fundamental properties from the spatial setting.
In the numerical experiments, we observed better convergence rates than theoretically proven when taking a high-resolution reconstruction as a surrogate ground truth.
We applied the regularizer to reconstruct both classical inverse problems in mathematical imaging and ECGI.
On the sample image, we matched the performance of the established piecewise constant discretization while using half the degrees of freedom.
In image inpainting, we also recovered linear structures in occluded regions, slightly improving quantitative reconstruction results.
In the transmembrane ECGI problem, where a piecewise constant discretization is not suitable, we outperformed both state-of-the-art Tikhonov and total variation regularization in three structurally different test cases.

Some limitations of the present analysis remain.
The convergence rates are established for the denoising problem; for ECGI, we prove existence of minimizers but rely on numerical evidence for the convergence of discrete solutions, and extending the rates to general linear forward operators remains open.
Moreover, the fully discrete result holds for the non-symmetric variant $\NSymTGV^2_\alpha$, since convergent finite element discretizations in the space of bounded deformation are only available in special settings~\cite{Bab23, Co19}.
In contrast to $\TV$, essential boundedness of the continuous minimizer does not follow from a truncation argument for $\TGV$ and is therefore assumed explicitly.
Additionally, the implemented functional is a lumped approximation of the analyzed integral formulation, so the theoretical rates apply to the latter only.
Due to the demanding spatiotemporal gradient computation and increased number of proximal maps, $\TGV^2_\alpha$ regularization is computationally more intensive than comparison methods.
However, this is justified by the consistently improved reconstruction quality on all test cases.

On the theoretical side, it remains open whether the observed rates of $h^{1/2}$ and $h$ are sharp; computing the continuous minimizer explicitly for a toy example would allow verifying the rates without a surrogate ground truth.
On the application side, natural next steps are three-dimensional anatomies, simulated cardiac activation sequences, as well as arrhythmias such as reentry waves.
We expect a key challenge there to be the invariance under additive constants of both the forward and inverse problems, which requires more sophisticated methods or models for an accurate solution.

	\section*{Acknowledgments}
	We want to thank Thomas Grandits and Thomas Pinetz for the helpful discussions during project meetings.
	\printbibliography
	
\end{document}